\documentclass[11pt]{article}
\usepackage{amsmath}
\usepackage{mathrsfs}
\usepackage{amssymb}
\usepackage{bm}
\usepackage{mathtools}
\usepackage{todonotes}

\usepackage{tikz}
\usetikzlibrary{decorations.markings}
\usetikzlibrary{calc}
\tikzset{directed/.style={decoration={markings,
mark=at position .6 with {\arrow[scale=1.6]{stealth}}},
postaction={decorate}}}
  
\usepackage{xcolor}   
\usepackage{geometry}
\usepackage[%
pdfstartview=FitH,
bookmarksopen=false,
colorlinks=true,
linkcolor=blue,
citecolor=blue,
urlcolor=blue]{hyperref}

\makeatletter
\def\@seccntDot{.}
\def\@seccntformat#1{\csname the#1\endcsname\@seccntDot\hskip 0.5em}
\renewcommand\section{\@startsection{section}{1}{\z@}%
{18\p@ \@plus 6\p@ \@minus 3\p@}%
{9\p@ \@plus 6\p@ \@minus 3\p@}%
{\large\bfseries\boldmath}}
\renewcommand\subsection{\@startsection{subsection}{2}{\z@}%
{12\p@ \@plus 6\p@ \@minus 3\p@}%
{3\p@ \@plus 6\p@ \@minus 3\p@}%
{\bfseries\boldmath}}
\renewcommand\subsubsection{\@startsection{subsubsection}{3}{\z@}%
{12\p@ \@plus 6\p@ \@minus 3\p@}%
{\p@}%
{\bfseries\boldmath}}
\makeatother

\newcommand{\A}{\mathcal{A}}

\usepackage{microtype}

\usepackage[amsmath,thmmarks]{ntheorem}
\theoremstyle{plain}
\newtheorem{theorem}{Theorem}[section]
\newtheorem{lemma}{Lemma}[section]
\newtheorem{corollary}{Corollary}[section]
\newtheorem{proposition}{Proposition}[section]

\theorembodyfont{\normalfont}
\newtheorem{definition}{Definition}[section]
\newtheorem{remark}{Remark}[section]
\newtheorem{example}{Example}[section]

\theoremstyle{nonumberplain}
\theoremseparator{}
\theoremsymbol{\ensuremath{\Box}}
\newtheorem{proof}{\it Proof.}

\numberwithin{equation}{section}
\allowdisplaybreaks
\begin{document}

\title{The $(p,q)$-spectral radii of $(r,s)$-directed hypergraphs}
\author{Lele Liu 
\thanks{Department of Mathematics, Shanghai University, Shanghai 200444, P.R. China
({\tt ahhylau@gmail.com}). The work is done when this author visited the University 
of South Carolina during September 2017--March 2019 under the support of the fund 
from the China Scholarship Council (CSC No. 201706890045).}
\and Linyuan Lu
\thanks{Department of Mathematics, University of South Carolina, Columbia, SC 29208, 
  USA ({\tt lu@math.sc.edu}). This author was supported in part by NSF grant DMS-1600811
  and ONR grant N00014-17-1-2842.}}

\maketitle

\begin{abstract}
An $(r,s)$-directed hypergraph is a directed hypergraph with $r$ vertices 
in tail and $s$ vertices in head of each arc. Let $G$ be an $(r,s)$-directed 
hypergraph. For any real numbers $p$, $q\geq 1$, we define the $(p,q)$-spectral 
radius $\lambda_{p,q}(G)$ as
\[
\lambda_{p,q}(G):=\max_{||\bm{x}||_p=||\bm{y}||_q=1}
\sum_{e\in E(G)}\Bigg(\prod_{u\in T(e)}x_u\Bigg)\Bigg(\prod_{v\in H(e)}y_v\Bigg),  
\]
where $\bm{x}=(x_1,\ldots,x_m)^{\mathrm{T}}$, $\bm{y}=(y_1,\ldots,y_n)^{\mathrm{T}}$
are real vectors; and $T(e)$, $H(e)$ are the tail and head of arc $e$, respectively.
We study some properties about $\lambda_{p,q}(G)$ including the bounds 
and the spectral relation between $G$ and its components. 

The $\alpha$-normal labeling method for uniform hypergraphs was introduced by 
Lu and Man in 2014. It is an effective method in studying the spectral radii 
of uniform hypergraphs. In this paper, we develop the $\alpha$-normal labeling 
method for calculating the $(p,q)$-spectral radii of $(r,s)$-directed hypergraphs.
Finally, some applications of $\alpha$-normal labeling method are given.
\vspace{2mm}

 \vspace{2mm}

\end{abstract}
\vspace{5mm}

\section{Introduction}
An $(r,s)$-directed hypergraph is a directed hypergraph with $r$ vertices in 
tail and $s$ vertices in head of each hyperarc. The purpose of this paper is 
to study the spectral properties of $(r,s)$-directed hypergraphs and develop 
a simple method to compute the spectral radii of directed hypergraphs.

Recall that an undirected hypergraph $H=(V,E)$ is a pair consisting of a vertex set 
$V$, and an edge set $E$ of subsets of $V$. A uniform hypergraph is a hypergraph in 
which each edge has the same size. In 2012, Cooper and Dutle \cite{Cooper2012} defined 
the spectra of uniform hypergraphs via eigenvalues of tensors introduced independently by 
Qi \cite{Qi2005} and Lim \cite{Lim2005}. Since then the spectral undirected hypergraph 
theory has been widely studied in \cite{Banerjee2017:General Spectral Radius,Khan2016,
LiShao2016,Nikiforov2014:Analytic Methods,Nikiforov2017:Symmetric Spectrum,Pearson2014,Shao2015}. 
For directed hypergraphs, in contrast, there are very few researches in spectral directed 
hypergraph theory so far. In 2016, Xie and Qi \cite{Qi2016:Spectral Directed Hypergraph} 
investigated the spectral properties of a specific kind of directed hypergraphs via 
tensors. Recently, Banerjee et al. \cite{Banerjee2017:Directed Spectral Radius} represent 
a general directed hypergraph by different connectivity tensors and study their spectral 
properties.

In this paper, we introduce a parameter $\lambda_{p,q}(G)$ for an $(r,s)$-directed 
hypergraph $G$ and real numbers $p$, $q\geq 1$ via multilinear function (see more 
details in Section \ref{sec2}), which called the $(p,q)$-spectral radius of $G$. 
Also, we give some properties about $\lambda_{p,q}(G)$.

In \cite{LuMan2016:Small Spectral Radius}, Lu and Man discovered a novel method for 
computing the spectral radii of uniform hypergraphs by introducing an $\alpha$-normal 
labeling method, which labels each corner of an edge by a positive number so that
the sum of the corner labels at any vertex is $1$ while the product of all corner 
labels at any edge is $\alpha$. This method has been proved by many researches 
\cite{KangLiu2016,LiuKang2016,BaiLu2017,OuyangQi2017,YuanSi2017,ZhangKang2017,
XiaoWang2018} to be a simple and effective method in the study of spectral 
radii of uniform hypergraphs. Recently, Liu and Lu \cite{LiuLu2018:p_spec} 
extend the $\alpha$-normal labeling method to the $p$-spectral radii of uniform 
hypergraphs. Motivated by the preceding work \cite{LuMan2016:Small Spectral Radius}
and \cite{LiuLu2018:p_spec}, in the present paper we develop the $\alpha$-normal 
labeling method for calculating the $(p,q)$-spectral radii of $(r,s)$-directed 
hypergraphs.

The remaining part of this paper is organized as follows. In Section \ref{sec2}, 
some preliminary definitions concerning directed hypergraphs and tensors are 
given. Moreover, we present the definition of $\lambda_{p,q}(G)$ for an 
$(r,s)$-directed hypergraph $G$. Section \ref{sec3} is dedicated to some 
basic properties of $\lambda_{p,q}(G)$. In Section \ref{sec4}, we develop 
the $\alpha$-normal labeling method for calculating the $\lambda_{p,q}(G)$ 
by constructing consistently $\alpha$-normal weighted incidence matrix for 
the target $(r,s)$-directed hypergraph. Also, we present a method for 
comparing the $\lambda_{p,q}(G)$ in terms of a particular value by 
constructing $\alpha$-subnormal or consistently $\alpha$-supernormal 
weighted incidence matrix. In Section \ref{sec5}, some applications
are given.

\section{Preliminaries}
\label{sec2}
In this section, we will review some basic notions of tensors and directed hypergraphs,
and present the definitions of $(r,s)$-directed hypergraph and its $(p,q)$-spectral radius.
For the basics on undirected hypergraphs we follow the traditions, as in \cite{Bretto2013}.

Let $\mathbb{R}$ be the field of real numbers and $\mathbb{R}^n$ the $n$-dimensional 
real space. Further, denote the nonnegative octant of $\mathbb{R}^n$ by $\mathbb{R}^n_+$. 
Given a vector $\bm{x}=(x_1,x_2,\ldots,x_n)^{\mathrm{T}}$ and a set 
$S\subseteq [n]:=\{1,2,\ldots,n\}$, write $\bm{x}|_S$ for the restriction of $\bm{x}$ 
over the set $S$. Also, we write $|\bm{x}|:=(|x_1|,|x_2|,\ldots,|x_n|)^{\mathrm{T}}$, and
$||\bm{x}||_p:=(|x_1|^p+|x_2|^p+\cdots+|x_n|^p)^{1/p}$. For any real number $p\geq 1$, 
denote $\mathbb{S}_p^{n-1}$ ($\mathbb{S}_{p,+}^{n-1}$, $\mathbb{S}_{p,++}^{n-1}$) the 
set of all (nonnegative, positive) real vectors $\bm{x}\in\mathbb{R}^n$ with $||\bm{x}||_p=1$. 

For positive integers $r$, $s$, $m$ and $n$, a real $(r,s)$-th order $(m\times n)$-dimensional 
rectangular tensor, or simply a real rectangular tensor, refers to a multidimensional array 
(also called hypermatrix) with entries $a_{i_1\cdots i_rj_1\cdots j_s}\in\mathbb{R}$ for all 
$i_1$, $i_2$,$\ldots$, $i_r\in[m]$ and $j_1$, $j_2$,\,$\ldots$, $j_s\in[n]$. Recently, the 
(weak) Perron--Frobenius theorem for rectangular tensors were studied in 
\cite{ChangQi2010,Qi2013:Singular value,YangYang2011,Gautier2016:Tensor Norm}. We say that 
$\mathcal{A}=(a_{i_1\cdots i_rj_1\cdots j_s})$ is {\em partially symmetric}, if 
$a_{i_1\cdots i_rj_1\cdots j_s}$ is invariant under any permutation of indices among 
$i_1$, $i_2$,\,$\ldots$, $i_r$ and any permutation of indices among $j_1$, $j_2$,$\ldots$, 
$j_s$, i.e.,
\[
a_{\pi(i_1\cdots i_r)\sigma(j_1\cdots j_s)}=a_{i_1\cdots i_rj_1\cdots j_s},~
\pi\in \mathfrak{S}_{r},~\sigma\in \mathfrak{S}_{s},    
\]
where $\mathfrak{S}_k$ is the permutation group of $k$ indices.

Let $\mathcal{A}=(a_{i_1\cdots i_rj_1\cdots j_s})$ be an $(r,s)$-th order 
$(m\times n)$-dimensional rectangular tensor. Denote
\begin{equation}\label{eq:Ax^ry^s}
\mathcal{A}\bm{x}^r\bm{y}^s:=\sum_{i_1,\ldots,i_r=1}^m\sum_{j_1,\ldots,j_s=1}^n
a_{i_1\cdots i_rj_1\cdots j_s}x_{i_1}\cdots x_{i_r}y_{j_1}\cdots y_{j_s}.
\end{equation}
A nonnegative $(r,s)$-th order $(m\times n)$-dimensional rectangular tensor
$\mathcal{A}=(a_{i_1\cdots i_rj_1\cdots j_s})$ is associated with an 
undirected bipartite graph $G(\mathcal{A})=(V,E(\mathcal{A}))$, the 
bipartition of which is $V=[m]\cup [n]$, and $(i_p,j_q)\in E(\mathcal{A})$
if and only if $a_{i_1\cdots i_rj_1\cdots j_s}>0$ for some $(r+s-2)$ indices 
$\{i_1,\ldots,i_r,j_1,\ldots,j_s\}\backslash\{i_p,j_q\}$. Following 
\cite{Friedland2013:Perron–Frobenius}, the tensor $\mathcal{A}$ is called 
{\em weakly irreducible} if the graph $G(\mathcal{A})$ is connected.

A directed hypergraph is a pair $G=(V(G),E(G))$, where $V(G)$ is a set of vertices, 
and $E(G)$ is a set of hyperarcs. A hyperarc or simply arc is an ordered pair, $e=(X,Y)$, 
of disjoint subsets of vertices, $X$ is the {\em tail} of $e$ while $Y$ is its {\em head}. 
We denote the number of arcs of $G$ by $|G|$. In the following, the tail and the head 
of an arc $e$ will be denoted by $T(e)$ and $H(e)$, respectively. Denote 
\[
T(G)=\bigcup_{e\in E(G)}T(e),~~H(G)=\bigcup_{e\in E(G)}H(e).
\]
For convenience, we always assume $|T(G)|=m$ and $|H(G)|=n$ throughout this paper.

The {\em in-degree} $d_v^-$ of a vertex $v$ in directed hypergraph $G$ is the number 
of arcs contained $v$ in head, and the {\em out-degree} $d_v^+$ of $v$ is the number 
of arcs contained $v$ in tail. The {\em degree} $d_v$ of a vertex $v$ is $d_v^++d_v^-$. 
The maximum in-degree and out-degree of $G$ are denoted by $\Delta^-$ and $\Delta^+$, 
respectively; likewise, the minimum in-degree and out-degree of $G$ are denoted by 
$\delta^-$ and $\delta^+$, respectively. Given two directed hypergraphs $G_1=(V_1,E_1)$ 
and $G_2=(V_2,E_2)$, if $V_1\subseteq V_2$ and $E_1\subseteq E_2$, then $G_1$ is called 
the {\em directed subhypergraph} of $G_2$, denoted by $G_1\subseteq G_2$. With any 
directed hypergraph $G$, we can associate an undirected hypergraph on the same vertex 
set simply by replacing each arc by an edge with the same vertices. This hypergraph 
is called the {\em underlying hypergraph} of $G$.

Now we introduce some new concepts for directed hypergraphs. In a directed hypergraph $G$, 
an {\em anadiplosis walk} of length $\ell$ is an alternating sequence of vertices and arcs 
$v_0e_1v_1e_2\cdots v_{\ell-1}e_{\ell}v_{\ell}$ such that either $v_i\in T(e_i)\cap T(e_{i+1})$ 
or $v_i\in H(e_i)\cap H(e_{i+1})$, $i\in[\ell-1]$. Furthermore, if $e_1$, $e_2$,\,$\ldots$, 
$e_{\ell}$ ($\ell\geq 2$) are all distinct arcs of $G$, and either 
$v_0=v_{\ell}\in T(e_1)\cap T(e_{\ell})$ or $v_0=v_{\ell}\in H(e_1)\cap H(e_{\ell})$, 
then this anadiplosis walk is called an {\em anadiplosis cycle}. An anadiplosis walk:
$v_0e_1v_1e_2\cdots v_{\ell-1}e_{\ell}v_{\ell}$ is called an {\em anadiplosis semi-cycle} 
if $e_1$, $e_2$,\,$\ldots$, $e_{\ell}$ ($\ell\geq 2$) are all distinct arcs of $G$ and
either $v_0=v_{\ell}\in T(e_1)\cap H(e_{\ell})$ or $v_0=v_{\ell}\in H(e_1)\cap T(e_{\ell})$.
A directed hypergraph $G$ is {\em anadiplosis connected} if there exists a $u$\,--\,$v$ 
anadiplosis walk for all $u\neq v$ in $V(G)$, and a $u$\,--\,$u$ anadiplosis semi-cycle 
for any $u\in T(G)\cap H(G)$. A maximal anadiplosis connected subhypergraph of $G$ is 
called an {\em anadiplosis component} of $G$.

\begin{remark}
In our definition above, vertex repetition is allowed in anadiplosis cycle and anadiplosis 
semi-cycle. In the following directed graph, $v_0e_1v_1e_2v_2e_3v_3e_4v_4e_5v_2e_6v_0$ 
is an anadiplosis cycle, and $v_2e_6v_0e_1v_1e_2v_2$ is an anadiplosis semi-cycle.
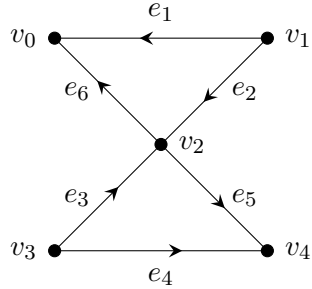
\begin{center}
    \begin{tikzpicture}
        \foreach \i in {1,3,5,7}
        {
        \coordinate (v\i) at (45*\i:2);
        \filldraw[fill=black] (v\i) circle (0.08);
        }
        \filldraw[fill=black] (0,0) circle (0.08);
        \draw[directed] (v1) node[right=1mm] {$v_1$}--(0,0) node[right=1mm] {$v_2$};
        \draw[directed] (0,0)--(v3);
        \draw[directed] (v1)--(v3) node[left=1mm] {$v_0$};
        \draw[directed] (v5) node[left=1mm] {$v_3$}--(0,0);
        \draw[directed] (0,0)--(v7) node[right=1mm] {$v_4$};
        \draw[directed] (v5)--(v7);
        \node[above=1mm] at ($(v1)!0.5!(v3)$) {$e_1$};
        \node[right=1mm] at ($(v1)!0.5!(0,0)$) {$e_2$};
        \node[left=1mm] at ($(0,0)!0.5!(v5)$) {$e_3$};
        \node[below=1mm] at ($(v5)!0.5!(v7)$) {$e_4$};
        \node[right=1mm] at ($(0,0)!0.5!(v7)$) {$e_5$};
        \node[left=1mm] at ($(0,0)!0.5!(v3)$) {$e_6$};
    \end{tikzpicture}
\end{center}
\end{remark}

\begin{definition}
A directed hypergraph $G$ is called an {\em $(r,s)$-directed hypergraph} if for any arc 
$e\in E(G)$, $|T(e)|=r$ and $|H(e)|=s$.
\end{definition}

\begin{definition}\label{def:adjacency tensor}
Let $G$ be an $(r,s)$-directed hypergraph. The {\em adjacency tensor} of $G$ is defined 
as an $(r,s)$-th order $(m\times n)$-dimensional rectangular tensor $\mathcal{A}(G)$, 
whose $(i_1,\ldots,i_r,j_1,\ldots,j_s)$-entry is $\frac{1}{r!s!}$
if $T(e)=\{i_1,i_2,\ldots,i_r\}$, $H(e)=\{j_1,j_2,\ldots,j_s\}$ for some $e\in E(G)$ and 
$0$ otherwise.
\end{definition}

By the definition above, the adjacency tensor of an $(r,s)$-directed hypergraph 
is partially symmetric. Given an $(r,s)$-directed hypergraph $G$, the polynomial 
form of $G$ is a multilinear function 
$P_G(\bm{x},\bm{y}): \mathbb{R}^m\times\mathbb{R}^n\to \mathbb{R}$ defined 
for any vectors $\bm{x}=(x_1,x_2,\ldots,x_m)^{\mathrm{T}}\in\mathbb{R}^m$, 
$\bm{y}=(y_1,y_2,\ldots,y_n)^{\mathrm{T}}\in\mathbb{R}^n$ as
\[
P_G(\bm{x},\bm{y}):=\mathcal{A}(G)\bm{x}^r\bm{y}^s=
\sum_{\substack{e\in E(G),\,T(e)=\{i_1,\ldots,i_r\} \\ H(e)=\{j_1,\ldots,j_s\}}}
x_{i_1}\cdots x_{i_r}y_{j_1}\cdots y_{j_s}.  
\]
  
We here give the definition of the $(p,q)$-spectral radius of an $(r,s)$-directed 
hypergraph.

\begin{definition}
Let $G$ be an $(r,s)$-directed hypergraph. For any $p$, $q\geq 1$, the 
{\em $(p,q)$-spectral radius} $\lambda_{p,q}(G)$ of $G$ is defined as
\begin{equation}\label{eq:(p,q)-spectral radius}
\lambda_{p,q}(G):=\max\left\{P_G(\bm{x},\bm{y}): \bm{x}\in\mathbb{S}^{m-1}_p, 
\bm{y}\in\mathbb{S}^{n-1}_q\right\}.
\end{equation}
In particular, if $p=2r$, $q=2s$, then $\lambda_{2r,2s}(G)$ is called the 
{\em spectral radius} of $G$, denoted by $\rho(G)$. That is
\begin{equation}\label{eq:spectral radius}
\rho(G):=\max\left\{P_G(\bm{x},\bm{y}): \bm{x}\in\mathbb{S}^{m-1}_{2r}, 
\bm{y}\in\mathbb{S}^{n-1}_{2s}\right\}.
\end{equation}
If $\bm{x}\in\mathbb{S}^{m-1}_p$ and $\bm{y}\in\mathbb{S}^{n-1}_q$ are two 
vectors such that $\lambda_{p,q}(G)=P_G(\bm{x},\bm{y})$, then $(\bm{x},\bm{y})$ 
will be called an {\em eigenpair} to $\lambda_{p,q}(G)$.
\end{definition}

Notice that $\mathbb{S}^{m-1}_p$ and $\mathbb{S}^{n-1}_q$ are compact sets, and 
$P_G(\bm{x},\bm{y})$ is continuous, thus $\lambda_{p,q}(G)$ is well defined. 
Clearly, equation \eqref{eq:(p,q)-spectral radius} is equivalent to
\begin{equation}\label{eq:Equi-max}
\lambda_{p,q}(G)=\max_{\bm{x}\neq0,\,\bm{y}\neq 0}
\frac{P_G(\bm{x},\bm{y})}{||\bm{x}||_p^r\cdot||\bm{y}||_q^s}.    
\end{equation}

\begin{remark}
Recall that $||\bm{x}||_{\infty}=\max_{1\leq i\leq m}\{|x_i|\}$ and
$||\bm{y}||_{\infty}=\max_{1\leq j\leq n}\{|y_j|\}$. Therefore
$\lim_{p,\,q\to\infty}\lambda_{p,q}(G)=|G|$. Denote $G_T$ the 
$r$-uniform hypergraph with $V(G_T)=T(G)$ and $\{i_1,i_2,\ldots,i_r\}\in E(G_T)$
if and only if $T(e)=\{i_1,i_2,\ldots,i_r\}$ for some arc $e\in E(G)$. 
Similarly, we can define the $s$-uniform hypergraph $G_H$. If $G_T$ 
has no repeated edges, then 
\[
\lim_{q\to\infty}\lambda_{p,q}(G)=\frac{\lambda^{(p)}(G_T)}{r},
\]
where $\lambda^{(p)}(G_T)$ is the scaled $p$-spectral radius of $G_T$
by removing a constant factor $(r-1)!$ from \cite{Keevash2014}.
If $G_H$ has no repeated edges, we also have
\[
\lim_{p\to\infty}\lambda_{p,q}(G)=\frac{\lambda^{(q)}(G_H)}{s}.
\]
\end{remark}

\begin{remark}
If $r=s=1$, the $(r,s)$-directed hypergraphs are exactly the directed graphs.
Let $G$ be a directed graph, $A=(a_{ij})$ be a $m\times n$ matrix with 
row indexed by the set $T(G)$ and column indexed by the set $H(G)$,
where $a_{ij}=1$ if $(i,j)$ is an arc of $G$, and $0$ otherwise. By 
\eqref{eq:spectral radius}, the spectral radius $\rho(G)$ of $G$ is 
exactly the largest singular value of $A$.   
\end{remark}
    
If $(\bm{x},\bm{y})\in\mathbb{S}^{m-1}_p\times\mathbb{S}^{n-1}_q$ is an eigenpair to 
$\lambda_{p,q}(G)$, then the vectors $\bm{x}'=|\bm{x}|$ and $\bm{y}'=|\bm{y}|$ also 
satisfy $||\bm{x}'||_p=||\bm{y}'||_q=1$ and so
\[
\lambda_{p,q}(G)=P_G(\bm{x},\bm{y})\leq P_G(\bm{x}',\bm{y}')\leq\lambda_{p,q}(G),  
\]
which yields $\lambda_{p,q}(G)=P_G(\bm{x}',\bm{y}')$. Therefore, there are 
always nonnegative vectors $\bm{x}$, $\bm{y}$ such that $||\bm{x}||_p=||\bm{y}||_q=1$ 
and $\lambda_{p,q}(G)=P_G(\bm{x},\bm{y})$. 

Let $(\bm{x},\bm{y})\in\mathbb{S}^{m-1}_{p,+}\times\mathbb{S}^{n-1}_{q,+}$ 
be an eigenpair to $\lambda_{p,q}(G)$. By Lagrange's method, there exists 
a $\mu$ such that for each $i\in T(G)$ with $x_i>0$,
\[
\frac{\partial P_G(\bm{x},\bm{y})}{\partial x_i}=
\sum_{e\in E(G),\,i\in T(e)}\Bigg(\prod_{u\in T(e),\,u\neq i}x_u\Bigg)
\Bigg(\prod_{v\in H(e)}y_v\Bigg)=p\mu x_i^{p-1}.    
\]
Multiplying the $i$-th equation by $x_i$ and adding them all, we have
\[
\sum_{i\in T(G)}\sum_{e\in E(G),\,i\in T(e)}\Bigg(\prod_{u\in T(e)}x_u\Bigg)
\Bigg(\prod_{v\in H(e)}y_v\Bigg)=p\mu\sum_{i\in T(G)} x_i^p=p\mu.      
\]
It follows that
\[
r\sum_{e\in E(G)}\Bigg(\prod_{u\in T(e)}x_u\Bigg)
\Bigg(\prod_{v\in H(e)}y_v\Bigg)=p\mu,
\]
which yields $r\lambda_{p,q}(G)=p\mu$. Therefore
\[
\sum_{e\in E(G),\,i\in T(e)}\Bigg(\prod_{u\in T(e),\,u\neq i}x_u\Bigg)
\Bigg(\prod_{v\in H(e)}y_v\Bigg)=r\lambda_{p,q}(G) x_i^{p-1}.    
\]
Similarly, for each $j\in H(G)$ with $y_j>0$, we have
\[
\sum_{e\in E(G),\,j\in H(e)}\Bigg(\prod_{u\in T(e)}x_u\Bigg)
\Bigg(\prod_{v\in H(e),\,v\neq j}y_v\Bigg)=s\lambda_{p,q}(G) y_j^{q-1}.
\]
Hence, we obtain the {\em weak eigenequations} of an $(r,s)$-directed hypergraph 
$G$ as follows:
\begin{equation}\label{eq:weakcharacteristic equation}
\begin{dcases}
\sum_{e\in E(G),\,T(e)=\{i,i_2,\ldots,i_r\} \atop H(e)=\{j_1,j_2,\ldots,j_s\}}
x_ix_{i_2}\cdots x_{i_r}y_{j_1}\cdots y_{j_s}=r\lambda_{p,q}(G) x_i^{p},~i\in T(G),\\[2mm]
\sum_{e\in E(G),\,T(e)=\{i_1,i_2,\ldots,i_r\} \atop H(e)=\{j,j_2,\ldots,j_s\}}
x_{i_1}\cdots x_{i_r}y_jy_{j_2}\cdots y_{j_s}=s\lambda_{p,q}(G) y_j^{q},~j\in H(G).
\end{dcases}
\end{equation}
If all $x_i>0$ and $y_j>0$, we can cancel one factor of $x_i$ and $y_j$, 
and obtain the {\em strong eigenequations} of an $(r,s)$-directed hypergraph 
$G$ as follows:
\begin{equation}\label{eq:strongcharacteristic equation}
\begin{dcases}
\sum_{e\in E(G),\,T(e)=\{i,i_2,\ldots,i_r\} \atop H(e)=\{j_1,j_2,\ldots,j_s\}}
x_{i_2}\cdots x_{i_r}y_{j_1}\cdots y_{j_s}=r\lambda_{p,q}(G) x_i^{p-1},~i\in T(G),\\[2mm]
\sum_{e\in E(G),\,T(e)=\{i_1,i_2,\ldots,i_r\} \atop H(e)=\{j,j_2,\ldots,j_s\}}
x_{i_1}\cdots x_{i_r}y_{j_2}\cdots y_{j_s}=s\lambda_{p,q}(G) y_j^{q-1},~j\in H(G).
\end{dcases}
\end{equation}

Before concluding this section, we list some inequalities which will be used in 
the sequel (see \cite{Hardy1988:Inequalities}). 
\begin{enumerate}
\item[(1)] {\bfseries (Generalized H\"older's inequality)}
Let $a_{ij}\geq 0$, $i\in [n]$, $j\in[m]$, be nonnegative real numbers, and 
$\alpha_1$, $\alpha_2$,\,$\ldots$, $\alpha_m$ be positive real numbers such 
that $\sum_{j=1}^m 1/\alpha_j=1$. Then
\begin{equation}\label{eq:Gener Holder}
\sum_{i=1}^n\Bigg(\prod_{j=1}^ma_{ij}\Bigg)\leq
\prod_{j=1}^m\Bigg(\sum_{i=1}^na_{ij}^{\alpha_j}\Bigg)^{1/\alpha_j}.    
\end{equation}
Equality holds if and only if either 
$\bm{x}^{(j)}:=\big(a_{1j}^{\alpha_j},a_{2j}^{\alpha_j},\ldots,a_{nj}^{\alpha_j}\big)^{\mathrm{T}}$, 
$j\in[m]$ are all proportional, or one of $\bm{x}^{(j)}$ is zero vector.  

\item[(2)] Let $a_{ij}\geq 0$, $i\in [n]$, $j\in[m]$. Suppose that $\alpha_1$,
$\alpha_2$,\,$\ldots$, $\alpha_m$ are positive real numbers such that 
$\sum_{j=1}^m 1/\alpha_j>1$, then
\begin{equation}\label{eq:GeqHolder}
\sum_{i=1}^n\Bigg(\prod_{j=1}^ma_{ij}\Bigg)\leq
\prod_{j=1}^m\Bigg(\sum_{i=1}^na_{ij}^{\alpha_j}\Bigg)^{1/\alpha_j}.    
\end{equation}
Equality holds if and only if either one of $\bm{x}^{(j)}$ is zero vector or all 
but one of each vector is zero, and in the latter case, those which are positive 
have the same rank.

\item[(3)] {\bfseries (Power Mean inequality)} Let $a_1$, $a_2$,\,$\ldots$, $a_n$ 
be positive real numbers, and $p$, $q$ be two nonzero real numbers such that $p<q$. Then
\begin{equation}\label{eq:PM}
\left(\frac1n\sum_{i=1}^na_i^p\right)^{1/p}\leq
\left(\frac1n\sum_{i=1}^na_i^q\right)^{1/q},
\end{equation}
with equality if and only if $a_1=a_2=\cdots=a_n$.

\item[(4)] {\bfseries (Jensen's inequality)}
Let $a_i\geq 0$, $i\in [n]$. If $0<q<p$, then
\begin{equation}\label{eq:Jensen inequality}
\Bigg(\sum_{i=1}^na_i^p\Bigg)^{1/p}\leq\Bigg(\sum_{i=1}^na_i^q\Bigg)^{1/q},   
\end{equation}
with equality holds if and only if all but one of $a_1$, $a_2$,\,$\ldots$, $a_n$ 
are zero.
\end{enumerate}

\section{Basic properties of $\lambda_{p,q}(G)$}
\label{sec3}
The first part of this section is devoted to some basic bounds about 
$\lambda_{p,q}(G)$. In the second part, we give an relation of 
$(p,q)$-spectral radius between $G$ and its anadiplosis components. 

Inspired from the ideas in \cite{Nikiforov2014:Analytic Methods},
we first consider $\lambda_{p,q}(G)$ as a function in $p$ and $q$ for a 
fixed $(r,s)$-directed hypergraph $G$. By changing the variables in 
\eqref{eq:(p,q)-spectral radius}, we obtain  
\[
\lambda_{p,q}(G)=\max_{||\bm{x}||_1=1,\,||\bm{y}||_1=1}
\sum_{e\in E(G),\,T(e)=\{i_1,\ldots,i_r\},\atop H(e)=\{j_1,\ldots,j_s\}}
(x_{i_1}\cdots x_{i_r})^{1/p}(y_{j_1}\cdots y_{j_s})^{1/q}.
\]
Assume $p$, $q$, $p'$, $q'\geq 1$ are positive real numbers. Applying the 
mean value theorem, we have
\[
(x_{i_1}\cdots x_{i_r})^{1/p}(y_{j_1}\cdots y_{j_s})^{1/q}-
(x_{i_1}\cdots x_{i_r})^{1/p'}(y_{j_1}\cdots y_{j_s})^{1/q'}\leq |p-p'|+|q-q'|.
\]
It follows that
\begin{align*}
	|\lambda_{p,q}(G)-\lambda_{p',q'}(G)| & \leq |G|(|p-p'|+|q-q'|)\\
	& \leq \sqrt{2}\,|G|\sqrt{(p-p')^2+(q-q')^2},
\end{align*}
which yields that $\lambda_{p,q}(G)$ is a continuous function in $p$ and $q$.
In Section \ref{sec5} we shall return to this topic, and give more properties on 
the function $\lambda_{p,q}(G)$.

\subsection{Some bounds of $\lambda_{p,q}(G)$}
By equation \eqref{eq:(p,q)-spectral radius}, $\lambda_{p,q}(G)$ is monotone 
with respect to arc addition.

\begin{proposition}\label{prop:H<G}
Let $G_1$ and $G_2$ be two $(r,s)$-directed hypergraphs, and $G_1\subseteq G_2$. 
Then $\lambda_{p,q}(G_1)\leq\lambda_{p,q}(G_2)$.
\end{proposition}  

The following is a simple corollary of \autoref{prop:H<G}.

 \begin{corollary}
 Let $G$ be an $(r,s)$-directed hypergraph with maximum out-degree $\Delta^+$
 and maximum in-degree $\Delta^-$. Then 
 \[
 \lambda_{p,q}(G)\geq\frac{1}{r^{r/p}s^{s/q}}\cdot
 \max\left\{(\Delta^+)^{1-((r-1)/p+s/q)},~(\Delta^-)^{1-(r/p+(s-1)/q)}\right\}.    
 \]
 \end{corollary} 

% \begin{proof}
% Assume $u$ is a vertex attaining the maximum out-degree of $G$, i.e., $d^+_u=\Delta^+$.  
% Now we construct two vectors $\bm{x}=(x_1,x_2,\ldots,x_m)^{\mathrm{T}}\in\mathbb{R}^m$, 
% $\bm{y}=(y_1,y_2,\ldots,y_n)^{\mathrm{T}}\in\mathbb{R}^n$ as follows:
% \begin{align*}
% x_i & =\begin{dcases}
%     \frac{1}{\sqrt[p]{r}}, & \text{if}\ i=u,\\
%     \frac{1}{\sqrt[p]{r\Delta^+}}, & \text{if}\ \{i,u\}\subseteq T(e)\ \text{for some arc}\ e,\\
%     0, & \text{otherwise},
%     \end{dcases}\\[2mm]
% y_j & =\begin{dcases}
%     \frac{1}{\sqrt[q]{s\Delta^+}}, & \text{if}\ u\in T(e),\,j\in H(e)\ \text{for some arc}\ e,\\
%     0, & \text{otherwise}.
%     \end{dcases}
% \end{align*}
% Clearly, $||\bm{x}||_p\leq 1$, $||\bm{y}||_q\leq 1$. By \eqref{eq:Equi-max} we see that
% \[
% \lambda_{p,q}(G)\geq\frac{P_G(\bm{x},\bm{y})}{||\bm{x}||_p^r\cdot||\bm{y}||_q^s}
% \geq\frac{1}{r^{r/p}s^{s/q}}(\Delta^+)^{1-((r-1)/p+s/q)}.
% \]
% Using the similar arguments we can also prove that
% \[
% \lambda_{p,q}(G)\geq\frac{1}{r^{r/p}s^{s/q}}(\Delta^-)^{1-(r/p+(s-1)/q)}.
% \]
% The result follows.   
% \end{proof} 

\begin{proposition}
Let $G$ be an $(r,s)$-directed hypergraph with minimum out-degree $\delta^+$
and minimum in-degree $\delta^-$. Then 
\[
\lambda_{p,q}(G)\geq |G|^{1-(r/p+s/q)}
\bigg(\frac{\delta^+}{r}\bigg)^{r/p}\bigg(\frac{\delta^-}{s}\bigg)^{s/q}.  
\]
\end{proposition}
        
\begin{proof}
Let $\bm{x}=m^{-1/p}(1,1,\ldots,1)^{\mathrm{T}}\in\mathbb{R}^m$,
$\bm{y}=n^{-1/q}(1,1,\ldots,1)^{\mathrm{T}}\in\mathbb{R}^n$.  
By \eqref{eq:(p,q)-spectral radius}, we have
\begin{align*}
\lambda_{p,q}(G)\geq P_G(\bm{x},\bm{y}) 
& =\sum_{e\in E(G)}\frac{1}{m^{r/p}n^{s/q}}
  =\frac{|G|}{m^{r/p}n^{s/q}}\\
& =\frac{|G|^{1-(r/p+s/q)}}{r^{r/p}s^{s/q}}
   \bigg(\frac{r|G|}{m}\bigg)^{r/p}\bigg(\frac{s|G|}{n}\bigg)^{s/q}\\
&  \geq |G|^{1-(r/p+s/q)}
   \bigg(\frac{\delta^+}{r}\bigg)^{r/p}\bigg(\frac{\delta^-}{s}\bigg)^{s/q},
\end{align*}
the last inequality follows from the fact:
\[
\sum_{v\in T(G)}d_v^+=r|G|,~
\sum_{v\in H(G)}d_v^-=s|G|.    
\]
The proof is completed. 
\end{proof}

\subsection{Connectedness of $(r,s)$-directed hypergraphs}
We first introduce the definition of bipartite split of a directed 
hypergraph, which play an important role in the study of anadiplosis 
connectedness.

\begin{definition}\label{def:bipartite split}
Let $G$ be a directed hypergraph, the {\em bipartite split} $\mathcal{B}(G)$ 
of $G$ is define as a bipartite directed hypergraph with the same arc set 
as $G$ and bipartition $V_T\,\dot\cup\,V_H$, where $V_T$ is a copy of 
$T(G)$, and $V_H$ is a copy of $H(G)$. 
\end{definition}

\begin{example}
Let $G$ be a directed graph obtained by giving an orientation to $K_4$, 
the bipartite split $\mathcal{B}(G)$ of $G$ is shown as follows:

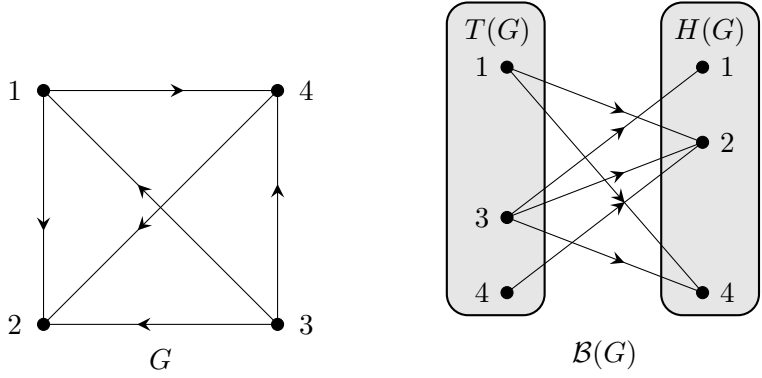
\begin{center}
\begin{tikzpicture}
\foreach \i in {1,3,5,7}
{
\coordinate (u\i) at (45*\i:2.2);
\filldraw[fill=black] (u\i) circle (0.08);
}
\draw[directed] (u3)--(u1) node[right=1.5mm] {$4$};
\draw[directed] (u7)--(u3) node[left=1.5mm] {$1$};
\draw[directed] (u3)--(u5) node[left=1.5mm] {$2$};
\draw[directed] (u7) node[right=1.5mm] {$3$}--(u5);
\draw[directed] (u7)--(u1);
\draw[directed] (u1)--(u5);
\node[below=2mm] at ($(u5)!0.5!(u7)$) {$G$};
\end{tikzpicture}\hspace{15mm}
\begin{tikzpicture}
\draw[fill=gray!20,line width=.8pt,rounded corners=8pt] (-0.8,0.7) rectangle (0.5,4.86);
\draw[fill=gray!20,line width=.8pt,rounded corners=8pt] (2.05,0.7) rectangle (3.35,4.86);
\node at (-0.12,4.476) {$T(G)$};
\node at (2.7,4.476) {$H(G)$};

\foreach \i in {1,2,4}
{\coordinate (u\i) at (0,\i);
\filldraw[fill=black] (u\i) circle (0.08);}

\foreach \j in {1,3,4}
{\coordinate (v\j) at (2.6,\j);
\filldraw[fill=black] (v\j) circle (0.08);}

\draw[directed] (u1) node[left=1mm] {$4$}--(v3);
\draw[directed] (u2) node[left=1mm] {$3$}--(v1);
\draw[directed] (u2)--(v3) node[right=1mm] {$2$};
\draw[directed] (u2)--(v4) node[right=1mm] {$1$};
\draw[directed] (u4)--(v1) node[right=1mm] {$4$};
\draw[directed] (u4) node[left=1mm] {$1$}--(v3);
\node[below=5mm] at ($(u1)!0.5!(v1)$) {$\mathcal{B}(G)$};
\end{tikzpicture}
\end{center}
\end{example}

Let $e$ be an arc of $G$, and $T(e)=\{i_1,\ldots,i_r\}$, $H(e)=\{j_1,\ldots,j_s\}$. 
We denote $\overline{e}$ the set consisting of $T(e)$ and $H(e)$, i.e., 
$\overline{e}=\{i_1,\ldots,i_r,j_1,\ldots,j_s\}$. For convenience, we denote 
$\overline{G}$ the underlying of $\mathcal{B}(G)$ in the sequel. The following 
lemma give an equivalent definition of anadiplosis connectedness.

\begin{lemma}\label{lem:connected iff underlying}
An $(r,s)$-directed hypergraph $G$ is anadiplosis connected if and only 
if $\overline{G}$ is connected. 
\end{lemma}

\begin{proof}
($\Longrightarrow$) Assume $G$ is anadiplosis connected, then for 
any $u\neq v$, there is a $u$\,--\,$v$ anadiplosis walk:
$(u=v_0)e_1v_1e_2v_2\cdots v_{\ell-1}e_{\ell}(v_{\ell}=v)$ in $G$. Clearly, 
$(u=v_0)\,\overline{e}_1v_1\overline{e}_2v_2\cdots v_{\ell-1}\overline{e}_{\ell}(v_{\ell}=v)$
is a $u$\,--\,$v$ walk in $\overline{G}$. Also, for each vertex $u\in V(\overline{G})$ with 
$u\in T(G)\cap H(G)$, there exists a $u$\,--\,$u$ anadiplosis semi-cycle:
$(u=v_0)e_1v_1e_2v_2\cdots v_{\ell-1}e_{\ell}(v_{\ell}=u)$ in $G$. Therefore 
$(u=v_0)\,\overline{e}_1v_1\overline{e}_2v_2\cdots v_{\ell-1}\overline{e}_{\ell}(v_{\ell}=u)$
is a $u$\,--\,$u$ walk in $\overline{G}$. Hence, $\overline{G}$ is connected.

($\Longleftarrow$) Assume $\overline{G}$ is connected. For any $u\neq v\in V(G)$, 
there is a $u$\,--\,$v$ walk: 
$(u=v_0)\,\overline{e}_1v_1\overline{e}_2v_2\cdots v_{\ell-1}\overline{e}_{\ell}(v_{\ell}=v)$
in $\overline{G}$. Obviously, $(u=v_0)e_1v_1e_2v_2\cdots v_{\ell-1}e_{\ell}(v_{\ell}=v)$ is 
an $u$\,--\,$v$ anadiplosis walk in $G$. For any $u\in T(G)\cap H(G)$, there is a $u$\,--\,$u$ walk:
$(u=v_0)\,\overline{e}_1v_1\overline{e}_2v_2\cdots v_{\ell-1}\overline{e}_{\ell}(v_{\ell}=v)$
in $\overline{G}$. Then $(u=v_0)e_1v_1e_2v_2\cdots v_{\ell-1}e_{\ell}(v_{\ell}=v)$ is 
a $u$\,--\,$u$ anadiplosis semi-cycle in $G$. Thus, $G$ is anadiplosis connected.
\end{proof}

\begin{lemma}\label{lem:weakly irreducible}
Let $G$ be an $(r,s)$-directed hypergraph. Then $\mathcal{A}(G)$ is weakly
irreducible if and only if $G$ is anadiplosis connected.
\end{lemma}

\begin{proof}
($\Longrightarrow$) Assume $\mathcal{A}(G)$ is weakly irreducible. Then its associated 
bipartite graph $G(\mathcal{A})$ is connected. For any $u$, $v\in V(\overline{G})$, since 
$G(\mathcal{A})$ is connected, there is a $u$\,--\,$v$ walk: 
$(u=v_0)\,e_1v_1\cdots v_{\ell-1}e_{\ell}(v_{\ell}=v)$ in $G(\mathcal{A})$. By the 
definition of $G(\mathcal{A})$, there must be $f_i\in E(\overline{G})$ such that 
$\{v_{i-1},v_i\}\subseteq f_i$, $i\in[\ell]$. Therefore 
$(u=v_0)\,f_1v_1\cdots v_{\ell-1}f_{\ell}(v_{\ell}=v)$ is a $u$\,--\,$v$ walk in 
$\overline{G}$. By \autoref{lem:connected iff underlying}, $G$ is anadiplosis connected.

($\Longleftarrow$) Assume $G$ is anadiplosis connected, according to 
\autoref{lem:connected iff underlying}, $\overline{G}$ is connected. 
For any vertices $u$, $v\in G(\mathcal{A})$, since $\overline{G}$ 
is connected, there is a $u$\,--\,$v$ walk: 
$(u=v_0)\,\overline{e}_1v_1\cdots v_{\ell-1}\overline{e}_{\ell}$
$(v_{\ell}=v)$ in $\overline{G}$. If there exists $i_0\in[\ell]$ such that
$\{v_{i_0-1},v_{i_0}\}\subseteq T(e_{i_0})$ (or $\{v_{i_0-1},v_{i_0}\}\subseteq H(e_{i_0})$), 
we pick any vertex $w\in H(e_{i_0})$ (or $w\in T(e_{i_0})$). Then 
\[
(u=v_0)\,\overline{e}_1v_1\cdots v_{i_0-1}\overline{e}_{i_0}w\overline{e}_{i_0}v_{i_0}
\cdots v_{\ell-1}\overline{e}_{\ell}\,(v_{\ell}=v)
\]
is also a $u$\,--\,$v$ walk in $\overline{G}$. Therefore we may assume 
$(u=v_0)\,\overline{e}_1v_1\cdots v_{\ell-1}\overline{e}_{\ell}(v_{\ell}=v)$ 
is a $u$\,--\,$v$ walk in $\overline{G}$ such that $\{v_{i-1},v_i\}\nsubseteq T(e_i)$
and $\{v_{i-1},v_i\}\nsubseteq H(e_i)$ for any $i\in[\ell]$. By the definition
of $G(\mathcal{A})$, $u=v_0\,,v_1,\ldots,v_{\ell-1},v_{\ell}=v$ is a $u$\,--\,$v$
walk in $G(\mathcal{A})$, i.e., $G(\mathcal{A})$ is connected, which implies 
$\mathcal{A}(G)$ is weakly irreducible.
\end{proof}

The following lemma establish an relation of the $(p,q)$-spectral radius between 
$G$ and its anadiplosis components.

\begin{lemma}\label{lem:spectral components}
Let $G$ be an $(r,s)$-directed hypergraph, $G_i$ be the anadiplosis components of 
$G$, $i=1$, $2$,\,$\ldots$, $k$. If $r/p+s/q\geq 1$, then
\[
\lambda_{p,q}(G)=\max_{1\leq i\leq k}\{\lambda_{p,q}(G_i)\}.    
\]
\end{lemma}
    
\begin{proof}
For any $i\in[k]$, let $(\bm{x}^{(i)},\bm{y}^{(i)})$ be an eigenpair 
corresponding to $\lambda_{p,q}(G_i)$. That is 
$\lambda_{p,q}(G_i)=P_{G_i}(\bm{x}^{(i)}, \bm{y}^{(i)})$. 
Now for any $u\in T(G)$, $v\in H(G)$, we construct two vectors 
$\bm{x}\in\mathbb{S}^{m-1}_p$, $\bm{y}\in\mathbb{S}^{n-1}_q$ as follows:
\begin{align*}
x_u & =\begin{cases}
    \big(\bm{x}^{(i)}\big)_u, & \text{if}\ u\in T(G_i),\\
    0, & \text{otherwise},
    \end{cases}\\    
y_v & =\begin{cases}
    \big(\bm{y}^{(i)}\big)_v, & \text{if}\ v\in H(G_i),\\
    0, & \text{otherwise}.
    \end{cases}   
\end{align*}
Therefore $\lambda_{p,q}(G)\geq P_{G}(\bm{x},\bm{y})=P_{G_i}(\bm{x}^{(i)},\bm{y}^{(i)})
=\lambda_{p,q}(G_i)$, which yields 
\[
\lambda_{p,q}(G)\geq\max_{1\leq i\leq k}\{\lambda_{p,q}(G_i)\}.
\]
    
On the other hand, let $(\bm{x},\bm{y})$ be an eigenpair corresponding to 
$\lambda_{p,q}(G)$. For any $u_i\in T(G_i)$, $v_i\in H(G_i)$, by the weak 
eigenequations \eqref{eq:weakcharacteristic equation} we have
\[
\begin{dcases}
    \sum_{e\in E(G),\,u_i\in T(e)}\Bigg(\prod_{u\in T(e)}x_u\Bigg)
    \Bigg(\prod_{v\in H(e)}y_v\Bigg)=r\lambda_{p,q}(G)x_{u_i}^p,\\[2mm]
    \sum_{e\in E(G),\,v_i\in H(e)}\Bigg(\prod_{u\in T(e)}x_u\Bigg)
    \Bigg(\prod_{v\in H(e)}y_v\Bigg)=s\lambda_{p,q}(G)y_{v_i}^q.
\end{dcases}
\]
According to \autoref{lem:connected iff underlying}, 
$\{e: e\in E(G), u_i\in T(e)\}=\{e: e\in E(G_i), u_i\in T(e)\}$. Also, we have 
$\{e: e\in E(G), v_i\in H(e)\}=\{e: e\in E(G_i), v_i\in H(e)\}$. Therefore
\begin{equation}\label{eq:characteristic equation for Gi}
    \begin{dcases}
        \sum_{e\in E(G_i),\,u_i\in T(e)}\Bigg(\prod_{u\in T(e)}x_u\Bigg)
        \Bigg(\prod_{v\in H(e)}y_v\Bigg)=r\lambda_{p,q}(G)x_{u_i}^p,\\[2mm]
        \sum_{e\in E(G_i),\,v_i\in H(e)}\Bigg(\prod_{u\in T(e)}x_u\Bigg)
        \Bigg(\prod_{v\in H(e)}y_v\Bigg)=s\lambda_{p,q}(G)y_{v_i}^q.
    \end{dcases}
\end{equation}
Summing both sides on $u_i\in T(G_i)$ and $v_i\in H(G_i)$, respectively, we obtain
\begin{equation}\label{eq:sum}
    \begin{dcases}
        \sum_{e\in E(G_i)}\Bigg(\prod_{u\in T(e)}x_u\Bigg)
        \Bigg(\prod_{v\in H(e)}y_v\Bigg)=\lambda_{p,q}(G)||\bm{x}|_{T(G_i)}||_p^p,\\[2mm]
        \sum_{e\in E(G_i)}\Bigg(\prod_{u\in T(e)}x_u\Bigg)
        \Bigg(\prod_{v\in H(e)}y_v\Bigg)=\lambda_{p,q}(G)||\bm{y}|_{H(G_i)}||_q^q.
    \end{dcases}
\end{equation}
Hence, $||\bm{x}|_{T(G_i)}||_p^p=||\bm{y}|_{H(G_i)}||_q^q$, $i\in [k]$. Now we choose an 
anadiplosis component $G_j$ such that $\bm{x}|_{T(G_j)}\neq 0$ and $\bm{y}|_{H(G_j)}\neq 0$.
It follows from \eqref{eq:Equi-max} and \eqref{eq:sum} that 
\begin{align*}
\lambda_{p,q}(G_j) & \geq
\frac{P_{G_j}\big(\bm{x}|_{T(G_j)},\bm{y}|_{H(G_j)}\big)}%
{||\bm{x}|_{T(G_j)}||_p^r\cdot||\bm{y}|_{H(G_j)}||_q^s}\\[2mm]  
& =\frac{\sum_{e\in E(G_j)}\left(\prod_{u\in T(e)}x_u\right)\left(\prod_{v\in H(e)}y_v\right)}%
{||\bm{x}|_{T(G_j)}||_p^r\cdot||\bm{y}|_{H(G_j)}||_q^s}\\[2mm]
& =\frac{\lambda_{p,q}(G)}{||\bm{x}|_{T(G_j)}||_p^{p(r/p+s/q-1)}}
\geq\lambda_{p,q}(G).
\end{align*}
Therefore $\max_{1\leq i\leq t}\lambda_{p,q}(G_i)\geq\lambda_{p,q}(G)$,
completing the proof.
\end{proof}
    
However, if $r/p+s/q<1$, we get a different statement as follows.

\begin{lemma}\label{lem:r/p+s/q<1-components}
Let $G$ be an $(r,s)$-directed hypergraph, $G_i$ be the anadiplosis components of 
$G$, $i=1$, $2$,\,$\ldots$, $k$. If $r/p+s/q<1$, then
\[
\lambda_{p,q}(G)=
\Bigg(\sum_{i=1}^k\big(\lambda_{p,q}(G_i)\big)^{1/(1-(r/p+s/q))}\Bigg)^{1-(r/p+s/q)}.    
\]
\end{lemma}  

\begin{proof}
Let $(\bm{x},\bm{y})\in\mathbb{S}^{m-1}_{p,+}\times\mathbb{S}^{n-1}_{q,+}$ be an 
eigenpair to $\lambda_{p,q}(G)$, and let $\bm{x}^{(i)}$, $\bm{y}^{(i)}$ be 
the restriction of $\bm{x}$, $\bm{y}$ to $T(G_i)$, $H(G_i)$, respectively.  
In the light of \eqref{eq:Equi-max},
\begin{align*}
\lambda_{p,q}(G)=P_G(\bm{x},\bm{y})=\sum_{i=1}^kP_{G_i}(\bm{x}^{(i)},\bm{y}^{i})
\leq\sum_{i=1}^k\lambda_{p,q}(G_i)||\bm{x}^{(i)}||_p^r\cdot ||\bm{y}^{(i)}||_q^s.
\end{align*}
Let $\alpha_1=1/(1-(r/p+s/q))$, $\alpha_2=p/r$, $\alpha_3=q/s$, we have
$1/\alpha_1+1/\alpha_2+1/\alpha_3=1$, and applying Generalized H\"older's inequality 
\eqref{eq:Gener Holder}, we obtain
\begin{align*}
\lambda_{p,q}(G) 
& \leq\Bigg(\sum_{i=1}^k\big(\lambda_{p,q}(G_i)\big)^{\alpha_1}\Bigg)^{1/\alpha_1}
  \Bigg(\sum_{i=1}^k\big(||\bm{x}^{(i)}||_p^r\big)^{\alpha_2}\Bigg)^{1/\alpha_2}
  \Bigg(\sum_{i=1}^k\big(||\bm{y}^{(i)}||_q^s\big)^{\alpha_3}\Bigg)^{1/\alpha_3}\\
& =\Bigg(\sum_{i=1}^k\big(\lambda_{p,q}(G_i)\big)^{\alpha_1}\Bigg)^{1/\alpha_1}
  \Bigg(\sum_{i=1}^k ||\bm{x}^{(i)}||_p^p\Bigg)^{r/p}
  \Bigg(\sum_{i=1}^k ||\bm{y}^{(i)}||_q^q\Bigg)^{s/q}\\
& =\Bigg(\sum_{i=1}^k\big(\lambda_{p,q}(G_i)\big)^{1/(1-(r/p+s/q))}\Bigg)^{1-(r/p+s/q)}.
\end{align*}

On the other hand, let $(\bm{x}^{(i)},\bm{y}^{(i)})$ be the eigenpair to $\lambda_{p,q}(G_i)$,
that is $\lambda_{p,q}(G_i)=P_{G_i}(\bm{x}^{(i)},\bm{y}^{(i)})$,~$i=1$, $2$,\,$\ldots$, $k$.
For simplicity, denote
\[
a_i=\frac{\big(\lambda_{p,q}(G_i)\big)^{\alpha_1}}{\sum_{i=1}^k\big(\lambda_{p,q}(G_i)\big)^{\alpha_1}},~
i=1,2,\ldots,k.   
\] 
Furthermore, we let
\begin{align*}
\bm{x} & =\big(a_1^{1/p}\bm{x}^{(1)},a_2^{1/p}\bm{x}^{(2)},\ldots,a_k^{1/p}\bm{x}^{(k)}\big)^{\mathrm{T}},\\
\bm{y} & =\big(a_1^{1/q}\bm{y}^{(1)},a_2^{1/q}\bm{y}^{(2)},\ldots,a_k^{1/q}\bm{y}^{(k)}\big)^{\mathrm{T}}.
\end{align*}
Clearly, $||\bm{x}||_p^p=||\bm{y}||_q^q=1$. By \eqref{eq:(p,q)-spectral radius} we see that
\begin{align*}
\lambda_{p,q}(G)\geq P_G(\bm{x},\bm{y}) 
& =\sum_{i=1}^kP_{G_i}\big(a_i^{1/p}\bm{x}^{(i)},a_i^{1/q}\bm{y}^{(i)}\big)\\
& =\sum_{i=1}^ka_i^{r/p+s/q}\cdot P_{G_i}\big(\bm{x}^{(i)},\bm{y}^{(i)}\big)\\
& =\sum_{i=1}^ka_i^{r/p+s/q}\cdot\lambda_{p,q}(G_i)\\
& =\Bigg(\sum_{i=1}^k\big(\lambda_{p,q}(G_i)\big)^{1/(1-(r/p+s/q))}\Bigg)^{1-(r/p+s/q)}.
\end{align*}  
The proof is completed.  
\end{proof}

We call the value $e:=\frac{r}{p}+\frac{s}{q}$ is the {\em eccentricity} of 
$\lambda_{p,q}(G)$, and refer $e<1$ as the elliptical phase, $e=1$ as the 
parabolic phase, and $e>1$ as the hyperbolic phase. The value $\lambda_{p,q}(G)$ 
behaves very different in three phases.

\section{The $\alpha$-normal labeling methods for  $(r,s)$-directed hypergraphs}
\label{sec4}
We begin this section with the following concept, which will be used frequently 
in the sequel. 

\begin{definition}
A {\em weighted incidence matrix} $B=(B(v,e))$ of a (directed) hypergraph $G$ 
is a $|V|\times |E|$ matrix such that for any $v\in V(G)$ and any $e\in E(G)$, 
the entry $B(v,e)>0$ if $v\in e$ and $B(v,e)=0$ if $v\notin e$.
\end{definition}

In \cite{LuMan2016:Small Spectral Radius}, Lu and Man discovered
the $\alpha$-normal labeling method for computing the spectral radii 
of uniform hypergraphs as follows.

\begin{theorem}[\cite{LuMan2016:Small Spectral Radius}]
\label{thm:uniform hypergraph normal label}
Let $H$ be a connected $k$-uniform hypergraph. Then the spectral radius of $H$ is
$\rho(H)$ if and only if there is a weighted incidence matrix $B=(B(v,e))$ satisfying
\begin{enumerate}
\item[$(1)$] $\sum_{e:\,v\in e}B(v,e)=1$, for any $v\in V(H)$;
\item[$(2)$] $\prod_{v\in e}B(v,e)=\alpha=(\rho(H))^{-k}$, for any $e\in E(H)$;
\item[$(3)$] $\prod_{i=1}^{\ell}\frac{B(v_{i-1},e_i)}{B(v_i,e_i)}=1$, for any cycle
$v_0e_1v_1e_2\cdots v_{\ell-1}e_{\ell}v_{\ell}\,(v_{\ell}=v_0)$.
\end{enumerate}
\end{theorem}

In our previous paper \cite{LiuLu2018:p_spec}, we generalized 
the $\alpha$-normal labeling method for computing the $p$-spectral radii of
$k$-uniform hypergraphs and found a number of applications for $p>k$. 

\begin{theorem}[\cite{LiuLu2018:p_spec}]\label{thm:p_spec}
Let $H$ be a $k$-uniform hypergraph and $p>k$. Then the $p$-spectral radius of $H$ is
$\lambda^{(p)}(H)$ if and only if there exist a weighted incidence matrix 
$B=(B(v,e))$ and edge weights $\{w(e)\}$ satisfying
\begin{enumerate}
\item[$(1)$] $\sum_{e\in E(G)}w(e)=1$; 

\item[$(2)$] $\sum_{e:\,v\in e}B(v,e)=1$, for any $v\in V(H)$; 
 
\item[$(3)$] $w(e)^{p-k}\cdot\prod_{v\in e}B(v,e)=\alpha=k^{p-k}/(\lambda^{(p)}(H))^p$, 
for any $e\in E(H)$;

\item[$(4)$] For any $v\in V(H)$ and $v\in e_i$, $i=1,2,\ldots,d$,
\[
\frac{w(e_1)}{B(v,e_1)}=\frac{w(e_2)}{B(v,e_2)}=\cdots=\frac{w(e_d)}{B(v,e_d)}.
\]
\end{enumerate}
\end{theorem}

The main focus of this section is to develop a similar method as 
\autoref{thm:uniform hypergraph normal label} (and \autoref{thm:p_spec}) 
for calculating $\lambda_{p,q}(G)$, as well as for comparing $\lambda_{p,q}(G)$ 
with a particular value. Before continuing, we need the following 
Perron--Frobenius theorem for rectangular tensors. We say an index 
$v$ is an {\em isolated vertex} for a rectangular tensor $\mathcal{A}$ 
if the $v$-th row is zero: $a_{vi_2\cdots i_rj_1\cdots j_s}\equiv 0$ or 
$a_{i_1\cdots i_rvj_2\cdots j_s}\equiv 0$.
  
\begin{theorem}[Lu, Yang, Zhao \cite{LYZ}]\label{thm:Perron–Frobenius theorem}
Suppose that $\A$ is an $(r,s)$-th order $(m\times n)$-dimensional nonnegative rectangular 
tensor with no isolated vertex.
\begin{enumerate}
\item[$(1)$] If $r/p+s/q<1$, then $\A$ has a unique positive eigenvalue-eigenvetors triple.
\item[$(2)$] If $(r-1)/p+s/q<1$, $r/p+(s-1)/q<1$, and $\A$ is partially symmetric and weakly 
irreducible, then $\A$ has a positive eigenvalue-eigenvetors triple. If further $r/p+s/q=1$, 
then $\A$ has a unique positive eigenvalue-eigenvetors triple.
\end{enumerate}
\end{theorem}

From \autoref{thm:Perron–Frobenius theorem} and \autoref{lem:weakly irreducible}, 
we have the following statement.

\begin{theorem}\label{thm:Perron–Frobenius for dihypergraphs}
Suppose that $G$ is an $(r,s)$-directed hypergraphs with no isolated vertex.
\begin{enumerate}
\item[$(1)$] If $r/p+s/q<1$, then $G$ has a unique positive eigenpair to $\lambda_{p,q}(G)$.
\item[$(2)$] If $r/p+s/q=1$, and $G$ is anadiplosis connected,
then $G$ has a unique positive eigenpair to $\lambda_{p,q}(G)$.
\end{enumerate}
\end{theorem}

We say that vertices $u$ and $v$ are {\em equivalent} in $G$, in writing $u\sim v$, 
if there exists an automorphism $\pi$ of $G$ such that $\pi(u)=v$.
The following is a direct corollary of \autoref{thm:Perron–Frobenius for dihypergraphs}.

\begin{corollary}\label{coro:equivalence}
Let $G$ be an $(r,s)$-directed hypergraph, and let $u$, $v\in T(G)$ 
$($or $u$, $v\in H(G)$$)$, $u\sim v$. Suppose that  
$(\bm{x},\bm{y})\in\mathbb{S}^{m-1}_{p,+}\times\mathbb{S}^{n-1}_{q,+}$ 
is an eigenpair to $\lambda_{p,q}(G)$. Then $x_u=x_v$ $($or $y_u=y_v)$ 
if one of the following holds: 
\begin{enumerate}
\item[$(1)$] $r/p+s/q<1$;
\item[$(2)$] $r/p+s/q=1$ and $G$ is anadiplosis connected.
\end{enumerate}
\end{corollary}

\subsection{Parabolic phase: $\frac{r}{p}+\frac{s}{q}=1$}

\begin{definition}\label{def:parabolic-normal}
An $(r,s)$-directed hypergraph $G$ is called {\em parabolic $\alpha$-normal} if 
$r/p+s/q=1$ and there exists a weighted incidence matrix $B$ satisfying
\begin{enumerate}
\item[(1)] $\displaystyle\sum_{e:\,u\in T(e)}B(u,e)=\sum_{e:\,v\in H(e)}B(v,e)=1$, 
for any $u\in T(G)$, $v\in H(G)$; 

\item[(2)] $\displaystyle \prod_{u\in T(e)}\big(B(u,e)\big)^{1/p}\cdot
\prod_{v\in H(e)}\big(B(v,e)\big)^{1/q}=\alpha$, for any $e\in E(G)$.
\end{enumerate}
Moreover, the weighted incidence matrix $B$ is called {\em parabolic consistent} if 
for any anadiplosis cycle $v_0e_1v_1e_2\cdots v_{\ell-1}e_{\ell}v_{\ell}$ 
$(v_{\ell}=v_0)$,
\[
\prod_{i=1}^{\ell}\frac{B(v_{i-1},e_i)}{B(v_i,e_i)}=1.
\]
\end{definition}

\begin{example}
Consider the following $(2,1)$-directed hypergraph $C_{\ell}^{(2,1)}$ with $\ell$ 
arcs ($\ell$ is even). Clearly, $C_{\ell}^{(2,1)}$ is an anadiplosis cycle. 
Here, black (white) node represents the vertex in tail (head).
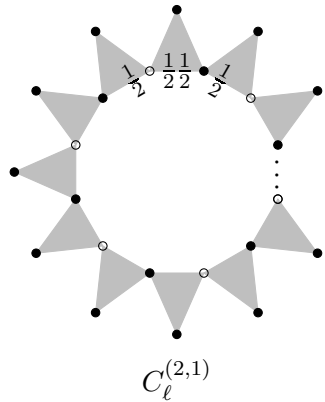
\begin{center}
\begin{tikzpicture}[scale=0.8]
\foreach \i in {30,60,...,330}
\filldraw[fill=gray!50,draw=gray!50] (\i:2.7)--(\i-15:1.7387)--(\i+15:1.7387)--(\i:2.7);

\foreach \i in {30,60,...,330}
\filldraw[fill=black] (\i:2.7) circle (0.07);

\foreach \i in {-15,45,...,345}
\draw (\i:1.7387) circle (0.07);

\foreach \i in {15,75,...,315}
\filldraw[fill=black] (\i:1.7387) circle (0.07);

\node[rotate=90] at ($(15:1.7387)!0.5!(-15:1.7387)$) {$\cdots$};
\node[right,rotate=-30] at (75:1.7387) {$\frac{1}{2}$};
\node[left] at (75:1.7387) {$\frac{1}{2}$};
\node[right] at (105:1.7387) {$\frac{1}{2}$};
\node[left,rotate=30] at (105:1.7387) {$\frac{1}{2}$};
\node[below=2mm] at (270:2.7) {$C_{\ell}^{(2,1)}$};
\end{tikzpicture}
\end{center}
We label the value $B(v,e)$ at vertex $v$ near the side of arc $e$. If $d_v=1$, 
then it has the trivial value $1$, and therefore we omit its labeling. If $d_v=2$, 
we let $B(v,e)=1/2$. It can checked that $C_{\ell}^{(2,1)}$ is parabolic 
consistently $2^{-(1/p+1/q)}$-normal.
\end{example}

\begin{lemma}\label{lem:iff}
Let $G$ be an $(r,s)$-directed hypergraph with $r/p+s/q=1$. If $G$ is anadiplosis 
connected, then the $(p,q)$-spectral radius of $G$ is $\lambda_{p,q}(G)$ if and only 
if $G$ is parabolic consistently $\alpha$-normal with 
\[
\alpha=\frac{1}{r^{r/p}s^{s/q} \lambda_{p,q}(G)}.    
\]
\end{lemma}

\begin{proof}
We first show that it is necessary. By \autoref{thm:Perron–Frobenius for dihypergraphs}, 
let $\bm{x}=(x_1,x_2,\ldots,x_m)^{\mathrm{T}}\in\mathbb{S}^{m-1}_{p,++}$ 
and $\bm{y}=(y_1,y_2,\ldots,y_n)^{\mathrm{T}}\in\mathbb{S}^{n-1}_{q,++}$ 
be an eigenpair to $\lambda_{p,q}(G)$. Define a weighted incidence matrix 
$B$ of $G$ as follows:
\begin{equation}\label{eq:B(v,e)}
B(v,e)=\begin{dcases}
\frac{\big(\prod_{u\in T(e)}x_u\big)\big(\prod_{u\in H(e)}y_u\big)}%
{r\lambda_{p,q}(G) x_v^p}, 
& \text{if}~v\in T(e),\\
\frac{\big(\prod_{u\in T(e)}x_u\big)\big(\prod_{u\in H(e)}y_u\big)}%
{s\lambda_{p,q}(G) y_v^q}, 
& \text{if}~v\in H(e),\\
0, & \text{otherwise}.
\end{dcases}
\end{equation}
For any $u\in T(G)$, $v\in H(G)$, by equation \eqref{eq:strongcharacteristic equation} we have
\[
\sum_{e:\,u\in T(e)}B(u,e)=
\frac{\sum_{e:\,u\in T(e)}\big(\prod_{v\in T(e)}x_v\big)\big(\prod_{v\in H(e)}y_v\big)}%
{r\lambda_{p,q}(G) x_u^p}=1
\]
and 
\[
\sum_{e:\,v\in H(e)}B(v,e)=
\frac{\sum_{e:\,v\in H(e)}\big(\prod_{u\in T(e)}x_u\big)\big(\prod_{u\in H(e)}y_u\big)}%
{s\lambda_{p,q}(G) y_v^q}=1.
\]
Also, for any $e\in E(G)$, it can be checked that
\[
\prod_{u\in T(e)}\big(B(u,e)\big)^{1/p}=
\frac{\Big(\prod_{v\in T(e)}x_v\prod_{v\in H(e)}y_v\Big)^{r/p}}%
{\big(r\lambda_{p,q}(G)\big)^{r/p}\prod_{v\in T(e)}x_v}
\]
and
\[
\prod_{v\in H(e)}\big(B(v,e)\big)^{1/q}=
\frac{\Big(\prod_{u\in T(e)}x_u\prod_{u\in H(e)}y_u\Big)^{s/q}}%
{\big(s\lambda_{p,q}(G)\big)^{s/q}\prod_{u\in H(e)}y_u}.
\]
It follows from $r/p+s/q=1$ that
\[
\prod_{u\in T(e)}\big(B(u,e)\big)^{1/p}\cdot\prod_{v\in H(e)}\big(B(v,e)\big)^{1/q}=
\frac{1}{r^{r/p}s^{s/q}\lambda_{p,q}(G)}=\alpha.
\]

To show that $B$ is parabolic consistent, for any anadiplosis cycle 
$v_0e_1v_1e_2\cdots v_{\ell-1}e_{\ell}v_{\ell}$ 
$(v_{\ell}=v_0)$, by the definition of anadiplosis cycle and \eqref{eq:B(v,e)} 
we conclude that
\[
\frac{B(v_i,e_{i+1})}{B(v_i,e_i)}=
\frac{\big(\prod_{u\in T(e_{i+1})}x_u\big)\big(\prod_{u\in H(e_{i+1})}y_u\big)}
{\big(\prod_{u\in T(e_i)}x_u\big)\big(\prod_{u\in H(e_i)}y_u\big)},~i\in [\ell-1].
\]
For short, we denote $Z(e):=\big(\prod_{u\in T(e)}x_u\big)\big(\prod_{u\in H(e)}y_u\big)$ 
for any $e\in E(G)$. Therefore
\[
\prod_{i=1}^{\ell}\frac{B(v_{i-1},e_i)}{B(v_i,e_i)}=
\frac{B(v_0,e_1)}{B(v_{\ell},e_{\ell})}\cdot
\prod_{i=1}^{\ell-1}\frac{B(v_i,e_{i+1})}{B(v_i,e_i)}
=\frac{Z(e_1)}{Z(e_{\ell})}\cdot\prod_{i=1}^{\ell-1}\frac{Z(e_{i+1})}{Z(e_i)}=1.
\]

Now we show that it is also sufficient. Assume that $B$ is a parabolic consistent 
$\alpha$-normal weighted incident matrix of $G$. For any nonnegative vectors 
$\bm{x}=(x_1,x_2,\ldots,x_m)^{\mathrm{T}}\in\mathbb{S}^{m-1}_{p,+}$ and 
$\bm{y}=(y_1,y_2,\ldots,y_n)^{\mathrm{T}}\in\mathbb{S}^{n-1}_{q,+}$, by 
H\"older's inequality and AM--GM inequality we have
\begin{align*}
P_G(\bm{x},\bm{y}) 
& =\sum_{e\in E(G)}\Bigg(\prod_{u\in T(e)}x_u\Bigg)\Bigg(\prod_{v\in H(e)}y_v\Bigg)\\
& =\frac{1}{\alpha}\sum_{e\in E(G)}\Bigg(\prod_{u\in T(e)}\big(B(u,e)\big)^{1/p}x_u\Bigg)
   \Bigg(\prod_{v\in H(e)}\big(B(v,e)\big)^{1/q}y_v\Bigg)\\
& \leq\frac{1}{\alpha}\Bigg(\sum_{e\in E(G)}\prod_{u\in T(e)}\big(B(u,e)\big)^{1/r}x_u^{p/r}\Bigg)^{r/p}\!
   \Bigg(\sum_{e\in E(G)}\prod_{v\in H(e)}\big(B(v,e)\big)^{1/s}y_v^{q/s}\Bigg)^{s/q}\\
& \leq\frac{1}{r^{r/p}s^{s/q}\alpha}\Bigg(\sum_{e\in E(G)}\sum_{u\in T(e)}B(u,e)x_u^p\Bigg)^{r/p}
   \left(\sum_{e\in E(G)}\sum_{v\in H(e)}B(v,e)y_v^q\right)^{s/q}\\
& =\frac{1}{r^{r/p}s^{s/q}\alpha}\cdot ||\bm{x}||_p^{r}\cdot ||\bm{y}||_q^{s}
  =\frac{1}{r^{r/p}s^{s/q}\alpha}.
\end{align*}
This inequality implies 
\begin{equation}\label{eq:sufficient conditions}
\lambda_{p,q}(G)\leq\frac{1}{r^{r/p}s^{s/q}\alpha}.
\end{equation}

The equality holds if $G$ is parabolic $\alpha$-normal and there is a nonzero solution 
$(\bm{x},\bm{y})$ to the following equations:
\begin{equation}\label{eq:solution}
rB(i_1,e)x_{i_1}^p=\cdots=rB(i_r,e)x_{i_r}^p=sB(j_1,e)y_{j_1}^q=\cdots 
=sB(j_s,e)y_{j_s}^q 
\end{equation}
for any $e\in E(G)$, $T(e)=\{i_1,i_2,\ldots,i_r\}$ and $H(e)=\{j_1,j_2,\ldots,j_s\}$. 
Fix a vertex $u_0\in T(G)$, now we consider any vertex $u\in T(G)$. Since $G$ is 
anadiplosis connected, there exists a $u_0$\,--\,$u$ anadiplosis walk: 
$u_0e_1u_1e_2\cdots u_{\ell-1}e_{\ell}u_{\ell}(u_{\ell}=u)$ in $G$. Define
\[
x_u^*=\left(\prod_{i=1}^{\ell}\frac{B(u_{i-1},e_i)}{B(u_i,e_i)}\right)^{1/p}x_{u_0}^*,    
\]
where $x_{u_0}^*$ is determined by the condition $||\bm{x}^*||_p=1$. 
Similarly, for any $v\in H(G)$, there is a $u_0$\,--\,$v$ anadiplosis walk: 
$u_0f_1v_1f_2\cdots v_{\ell'-1}f_{\ell'}v_{\ell'}(v_{\ell'}=v)$ in $G$. Define
\[
y_v^*=\left(\frac{r}{s}\prod_{i=1}^{\ell'}\frac{B(v_{i-1},f_i)}{B(v_i,f_i)}\right)^{1/q}
\left(x_{u_0}^*\right)^{p/q}.    
\] 
The consistent condition guarantees that $x_v^*$ and $y_v^*$ are independent of the choice 
of the anadiplosis walk. It is easy to check that $(\bm{x}^*,\bm{y}^*)$ is a solution of 
\eqref{eq:solution}, and
\begin{align*}
\sum_{v\in H(G)}(y_v^*)^q & =\sum_{e\in E(G)}\sum_{v\in H(e)}B(v,e)(y_v^*)^q\\
& =\sum_{e\in E(G)}\sum_{u\in T(e)}B(u,e)(x_u^*)^p\\
& =\sum_{u\in T(G)}(x_u^*)^p=1.    
\end{align*}
Therefore $\lambda_{p,q}(G)=r^{-r/p}s^{-s/q}\alpha^{-1}$.
The proof is completed.
\end{proof}

\begin{remark}
According to the proof of \autoref{lem:iff}, equation \eqref{eq:solution} is a 
sufficient condition for the equality holding in \eqref{eq:sufficient conditions}. 
We remark that it is also a necessary condition. That is, if $B$ is a parabolic 
consistently normal labeling of $G$, and
$(\bm{x},\bm{y})\in\mathbb{S}^{m-1}_{p,++}\times\mathbb{S}^{n-1}_{q,++}$ is an 
eigenpair to $\lambda_{p,q}(G)$, then \eqref{eq:solution} holds. Indeed, we assume 
\[
B(i_1,e)x_{i_1}^p=\cdots=B(i_r,e)x_{i_r}^p=cB(j_1,e)y_{j_1}^q
=\cdots =cB(j_s,e)y_{j_s}^q. 
\]
By $||\bm{x}||_p=||\bm{y}||_q=1$, we have
\begin{align*}
1=\sum_{u\in T(G)}x_u^p & =\sum_{e\in E(G)}\sum_{u\in T(e)}B(u,e)x_u^p\\
& =\sum_{e\in E(G)}\sum_{v\in H(e)}\frac{rc}{s}B(v,e)y_v^q\\
& =\frac{rc}{s}\sum_{v\in H(G)}y_v^q=\frac{rc}{s},
\end{align*}
which yields that $rc=s$.
\end{remark}

In what follows, we give a method for comparing the $(p,q)$-spectral radius with 
a particular value. It is convenient to introduce the following concepts.

\begin{definition}\label{def:parabolic-subnormal}
An $(r,s)$-directed hypergraph $G$ is called {\em parabolic $\alpha$-subnormal} 
if $r/p+s/q=1$ and there exists a weighted incidence matrix $B$ satisfying
\begin{enumerate}
\item[(1)] $\displaystyle\sum_{e:\,u\in T(e)}B(u,e)\leq 1$,~~
$\displaystyle\sum_{e:\,v\in H(e)}B(v,e)\leq1$, for any 
$u\in T(G)$, $v\in H(G)$; 

\item[(2)] $\displaystyle \prod_{u\in T(e)}\big(B(u,e)\big)^{1/p}\cdot
\prod_{v\in H(e)}\big(B(v,e)\big)^{1/q}\geq\alpha$, for any $e\in E(G)$.
\end{enumerate}
Moreover, $G$ is called {\em strictly parabolic $\alpha$-subnormal} if it is 
parabolic $\alpha$-subnormal but not parabolic $\alpha$-normal.
\end{definition}

Here is an example of parabolic $\alpha$-subnormal directed hypergraph.

\begin{example}
Consider the following $(2,1)$-directed hypergraph $P_{\ell}^{(2,1)}$. By 
labeling $P_{\ell}^{(2,1)}$ as follows:
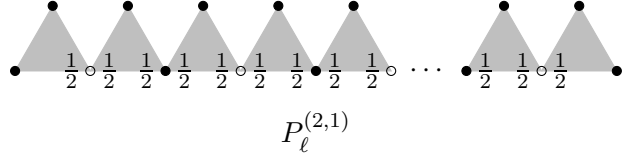
\begin{center}
\begin{tikzpicture}
\foreach \i in {0,1,...,8}
{\coordinate (u\i) at (\i,0);
\coordinate (v\i) at ($(60:1)+(\i,0)$);}
\foreach \j in {0,1,...,4,6}
\filldraw[fill=gray!50,draw=gray!50] (u\j)--(v\j)--($(u\j)+(1,0)$)--(u\j);

\foreach \i in {1,3,5}
\draw (u\i) circle (0.06);
\foreach \i in {0,2,4,6}
\filldraw[fill=black] (u\i) circle (0.06);

\foreach \i in {0,1,2,3,4,6}
\filldraw[fill=black] (v\i) circle (0.06);

\filldraw[fill=gray!50,draw=gray!50] (7,0)--(v7)--(u8)--(7,0);
\draw (7,0) circle (0.06);
\filldraw[fill=black] (u8) circle (0.06);
\filldraw[fill=black] (v7) circle (0.06);
\node at ($(u5)!0.5!(u6)$) {$\cdots$};

\foreach \i in {1,2,3,4,5,7}
\node[anchor=east] at (u\i) {$\frac{1}{2}$};
\foreach \i in {1,2,3,4,6,7}
\node[anchor=west] at (u\i) {$\frac{1}{2}$};
\node[below=4mm] at ($(u0)!0.5!(u8)$) {$P_{\ell}^{(2,1)}$};
\end{tikzpicture}
\end{center}
We can check that $P_{\ell}^{(2,1)}$ is parabolic $2^{-(1/p+1/q)}$-subnormal.
\end{example}

\begin{lemma}\label{lem:parabolic-subnormal}
Let $G$ be an $(r,s)$-directed hypergraph. If $G$ is parabolic $\alpha$-subnormal, 
then the $(p,q)$-spectral radius of $G$ satisfies
\[
\lambda_{p,q}(G)\leq\frac{1}{r^{r/p}s^{s/q}\alpha}.   
\]
\end{lemma}
    
\begin{proof}
For any nonnegative vectors $\bm{x}=(x_1,x_2,\ldots,x_m)^{\mathrm{T}}\in\mathbb{S}^{m-1}_{p,+}$ 
and $\bm{y}=(y_1,y_2,\ldots,y_n)^{\mathrm{T}}\in\mathbb{S}^{n-1}_{q,+}$, by H\"older's inequality 
and AM--GM inequality, we deduce that
\begin{align*}
P_G(\bm{x},\bm{y}) 
& \leq\frac{1}{\alpha}\sum_{e\in E(G)}\Bigg(\prod_{u\in T(e)}\big(B(u,e)\big)^{1/p}x_u\Bigg)
  \Bigg(\prod_{v\in H(e)}\big(B(v,e)\big)^{1/q}y_v\Bigg)\\
& \leq\frac{1}{\alpha}\Bigg(\sum_{e\in E(G)}\prod_{u\in T(e)}\big(B(u,e)\big)^{1/r}x_u^{p/r}\Bigg)^{r/p}
  \Bigg(\sum_{e\in E(G)}\prod_{v\in H(e)}\big(B(v,e)\big)^{1/s}y_v^{q/s}\Bigg)^{s/q}\\
& \leq\frac{1}{r^{r/p}s^{s/q} \alpha}\Bigg(\sum_{e\in E(G)}\sum_{u\in T(e)}B(u,e)x_u^p\Bigg)^{r/p}
  \left(\sum_{e\in E(G)}\sum_{v\in H(e)}B(v,e)y_v^q\right)^{s/q}\\
& \leq\frac{1}{r^{r/p}s^{s/q}\alpha}\cdot ||\bm{x}||_p^r\cdot ||\bm{y}||_q^s
=\frac{1}{r^{r/p}s^{s/q}\alpha},
\end{align*}
which implies $\lambda_{p,q}(G)\leq r^{-r/p}s^{-s/q}\alpha^{-1}$. When $G$ is 
strictly parabolic $\alpha$-subnormal, this inequality is strict, and therefore 
$\lambda_{p,q}(G)<r^{-r/p}s^{-s/q}\alpha^{-1}$.
\end{proof}

\begin{definition}\label{def:parabolic-supernormal}
An $(r,s)$-directed hypergraph $G$ is called {\em parabolic $\alpha$-supernormal} 
if $r/p+s/q=1$ and there exists a weighted incidence matrix $B$ satisfying
\begin{enumerate}
\item[(1)] $\displaystyle\sum_{e:\,u\in T(e)}B(u,e)\geq 1$,~~
$\displaystyle\sum_{e:\,v\in H(e)}B(v,e)\geq 1$,
for any $u\in T(G)$, $v\in H(G)$;

\item[(2)] $\displaystyle \prod_{u\in T(e)}\big(B(u,e)\big)^{1/p}\cdot
\prod_{v\in H(e)}\big(B(v,e)\big)^{1/q}\leq\alpha$,
for any $e\in E(G)$.
\end{enumerate}
Moreover, $G$ is called {\em strictly parabolic $\alpha$-supernormal} if it is 
parabolic $\alpha$-supernormal but not parabolic $\alpha$-normal.
\end{definition}

\begin{lemma}\label{lem:parabolic-supernormal}
Let $G$ be an $(r,s)$-directed hypergraph. If $G$ is parabolic consistently 
$\alpha$-supernormal, then the $(p,q)$-spectral radius of $G$ satisfies
\[
\lambda_{p,q}(G)\geq\frac{1}{r^{r/p}s^{s/q}\alpha}.  
\]
\end{lemma}

\begin{proof}
The parabolic consistent condition implies that there exist nonnegative vectors 
$\bm{x}=(x_1,x_2,\ldots,x_m)^{\mathrm{T}}\in\mathbb{S}^{m-1}_{p,+}$
and $\bm{y}=(y_1,y_2,\ldots,y_n)^{\mathrm{T}}\in\mathbb{S}^{n-1}_{q,+}$ 
satisfying \eqref{eq:solution}. Therefore
\begin{align*}
P_G(\bm{x},\bm{y}) 
&  \geq \frac{1}{\alpha}\sum_{e\in E(G)}\Bigg(\prod_{u\in T(e)}\big(B(u,e)\big)^{1/p}x_u\Bigg)
   \Bigg(\prod_{v\in H(e)}\big(B(v,e)\big)^{1/q}y_v\Bigg)\\
&  =\frac{1}{\alpha}\Bigg(\sum_{e\in E(G)}\prod_{u\in T(e)}\big(B(u,e)\big)^{1/r}x_u^{p/r}\Bigg)^{r/p}
   \Bigg(\sum_{e\in E(G)}\prod_{v\in H(e)}\big(B(v,e)\big)^{1/s}y_v^{q/s}\Bigg)^{s/q}\\
&  =\frac{1}{r^{r/p}s^{s/q}\alpha}\Bigg(\sum_{e\in E(G)}\sum_{u\in T(e)}B(u,e)x_u^p\Bigg)^{r/p}
   \left(\sum_{e\in E(G)}\sum_{v\in H(e)}B(v,e)y_v^q\right)^{s/q}\\
&  \geq\frac{1}{r^{r/p}s^{s/q}\alpha}\cdot ||\bm{x}||_p^{r}\cdot ||\bm{y}||_q^{s}
   =\frac{1}{r^{r/p}s^{s/q}\alpha},
\end{align*}
which implies $\lambda_{p,q}(G)\geq r^{-r/p}s^{-s/q}\alpha^{-1}$. When $G$ is 
strictly parabolic $\alpha$-supernormal, this inequality is strict, and therefore 
$\lambda_{p,q}(G)>r^{-r/p}s^{-s/q}\alpha^{-1}$.
\end{proof}

\subsection{Elliptic phase: $\frac{r}{p}+\frac{s}{q}<1$}
Given an $(r,s)$-directed hypergraph $G$, for each arc $e\in E(G)$, we put a weight 
$w(e)>0$ on $e$. We now introduce the following concepts.

\begin{definition}\label{def:elliptic normal}
An $(r,s)$-directed hypergraph $G$ is called {\em elliptic $\alpha$-normal} if $r/p+s/q<1$ 
and there exist a weighted incidence matrix $B$ and $\{w(e)\}$ satisfying
\begin{enumerate}
\item[(1)] $\displaystyle\sum_{e\in E(G)}w(e)=1$;

\item[(2)] $\displaystyle\sum_{e:\,u\in T(e)}B(u,e)=\sum_{e:\,v\in H(e)}B(v,e)=1$, 
for any $u\in T(G)$, $v\in H(G)$; 

\item[(3)] $\displaystyle w(e)^{1-(r/p+s/q)}\prod_{u\in T(e)}\big(B(u,e)\big)^{1/p}\cdot
\prod_{v\in H(e)}\big(B(v,e)\big)^{1/q}=\alpha$, for any $e\in E(G)$.
\end{enumerate}
Moreover, the weighted incidence matrix $B$ and $\{w(e)\}$ are called {\em elliptic consistent} 
if for any $u\in T(G)$, $v\in H(G)$, $e_1$,\,$\ldots$, $e_d$ and $f_1$,\,$\ldots$, $f_h$ 
are arcs contained $u$ and $v$ in tail and head, respectively,
\[
\frac{w(e_1)}{B(u,e_1)}=\cdots=\frac{w(e_d)}{B(u,e_d)},~~
\frac{w(f_1)}{B(v,f_1)}=\cdots=\frac{w(f_h)}{B(v,f_h)}.     
\]
\end{definition}
   
\begin{lemma}\label{lem:elliptic-normal}
Let $G$ be an $(r,s)$-directed hypergraph with $r/p+s/q<1$. Then the $(p,q)$-spectral radius 
of $G$ is $\lambda_{p,q}(G)$ if and only if $G$ is elliptic consistently $\alpha$-normal with 
\[
\alpha=\frac{1}{r^{r/p}s^{s/q}\lambda_{p,q}(G)}.
\]
\end{lemma}
   
\begin{proof}
We first show that it is necessary. By \autoref{thm:Perron–Frobenius for dihypergraphs}, 
let $\bm{x}=(x_1,x_2,\ldots,x_m)^{\mathrm{T}}\in\mathbb{S}^{m-1}_{p,++}$ and 
$\bm{y}=(y_1,y_2,\ldots,y_n)^{\mathrm{T}}\in\mathbb{S}^{n-1}_{q,++}$ be the eigenpair 
to $\lambda_{p,q}(G)$. Define a weighted incidence matrix $B=(B(v,e))$ and $\{w(e)\}$ 
as follows:
\begin{align}\label{eq:elliptic B}
B(v,e) & =\begin{dcases}
\frac{\big(\prod_{u\in T(e)}x_u\big)\big(\prod_{u\in H(e)}y_u\big)}{r\lambda_{p,q}(G) x_v^p}, 
& \text{if}~v\in T(e),\\
\frac{\big(\prod_{u\in T(e)}x_u\big)\big(\prod_{u\in H(e)}y_u\big)}{s\lambda_{p,q}(G) y_v^q}, 
& \text{if}~v\in H(e),\\
0, & \text{otherwise}.
\end{dcases}\\[2mm]\label{eq:w(e)}
w(e) & =\frac{\big(\prod_{u\in T(e)}x_u\big)\big(\prod_{u\in H(e)}y_u\big)}{\lambda_{p,q}(G)}.  
\end{align}
For any $u\in T(G)$, $v\in H(G)$, by \eqref{eq:strongcharacteristic equation} we see that
\[
\sum_{e:\,u\in T(e)}B(u,e)=
\frac{\sum_{e:\,u\in T(e)}\big(\prod_{v\in T(e)}x_v\big)\big(\prod_{v\in H(e)}y_v\big)}%
{r\lambda_{p,q}(G) x_u^p}=1
\]
and 
\[
\sum_{e:\,v\in H(e)}B(v,e)=
\frac{\sum_{e:\,v\in H(e)}\big(\prod_{u\in T(e)}x_u\big)\big(\prod_{u\in H(e)}y_u\big)}%
{s\lambda_{p,q}(G) y_v^q}=1.
\]
Also, we have
\[
\sum_{e\in E(G)}w(e)=\frac{\sum_{e\in E(G)}\big(\prod_{u\in T(e)}x_u\big)%
\big(\prod_{u\in H(e)}y_u\big)}{\lambda_{p,q}(G)}=1.    
\]
Therefore, items (1) and (2) of \autoref{def:elliptic normal} are verified. 
For the item (3), we check that
\[
w(e)^{1-(r/p+s/q)}\prod_{u\in T(e)}\big(B(u,e)\big)^{1/p}\cdot\prod_{v\in H(e)}\big(B(v,e)\big)^{1/q}
=\frac{1}{r^{r/p}s^{s/q}\lambda_{p,q}(G)}=\alpha.    
\]

To show that $B$ and $\{w(e)\}$ are consistent, for any $u\in T(G)$, $v\in H(G)$, 
according to \eqref{eq:elliptic B} and \eqref{eq:w(e)} we see that
\[
\frac{w(e_1)}{B(u,e_1)}=\cdots=\frac{w(e_d)}{B(u,e_d)}=rx_u^p,~~
\frac{w(f_1)}{B(v,f_1)}=\cdots=\frac{w(f_h)}{B(v,f_h)}=sy_v^q.   
\]
   
Now we show that it is also sufficient. Assume that $G$ is elliptic consistently $\alpha$-normal
with weighted incident matrix $B$ and $\{w(e)\}$. Denote 
\[
\alpha_1=\frac{1}{1-(r/p+s/q)},~
\alpha_2=\frac{p}{r},~
\alpha_3=\frac{q}{s}.    
\]
Clearly, $1/\alpha_1+1/\alpha_2+1/\alpha_3=1$. For any nonnegative vectors 
$\bm{x}=(x_1,x_2,\ldots,x_m)^{\mathrm{T}}\in\mathbb{S}^{m-1}_{p,+}$ and 
$\bm{y}=(y_1,y_2,\ldots,y_n)^{\mathrm{T}}\in\mathbb{S}^{n-1}_{q,+}$, by Generalized 
H\"older's inequality \eqref{eq:Gener Holder} and AM--GM inequality, we have
\begin{align*}
P_G(\bm{x},\bm{y}) 
& =\sum_{e\in E(G)}\Bigg(\prod_{u\in T(e)}x_u\Bigg)\Bigg(\prod_{v\in H(e)}y_v\Bigg)\\
& =\frac{1}{\alpha}\sum_{e\in E(G)}\Bigg(w(e)^{1-(r/p+s/q)}\cdot\prod_{u\in T(e)}\big(B(u,e)\big)^{1/p}x_u
   \cdot\prod_{v\in H(e)}\big(B(v,e)\big)^{1/q}y_v\Bigg)\\
& \leq\frac{1}{\alpha}\Bigg(\sum_{e\in E(G)}w(e)^{[1-(r/p+s/q)]\alpha_1}\Bigg)^{1/\alpha_1}
  \Bigg(\sum_{e\in E(G)}\prod_{u\in T(e)}\big(B(u,e)\big)^{\alpha_2/p}x_u^{\alpha_2}\Bigg)^{1/\alpha_2}\\
&~~~\,\times\Bigg(\sum_{e\in E(G)}\prod_{v\in H(e)}\big(B(v,e)\big)^{\alpha_3/q}y_v^{\alpha_3}\Bigg)^{1/\alpha_3}\\
& =\frac{1}{\alpha}
  \Bigg(\sum_{e\in E(G)}\prod_{u\in T(e)}\big(B(u,e)\big)^{1/r}x_u^{p/r}\Bigg)^{r/p}
  \Bigg(\sum_{e\in E(G)}
  \prod_{v\in H(e)}\big(B(v,e)\big)^{1/s}y_v^{q/s}\Bigg)^{s/q}\\
& \leq\frac{1}{r^{r/p}s^{s/q}\alpha}\Bigg(\sum_{e\in E(G)}\sum_{u\in T(e)}B(u,e)x_u^p\Bigg)^{r/p}
  \Bigg(\sum_{e\in E(G)}\sum_{v\in H(e)}B(v,e)y_v^q\Bigg)^{s/q}\\
& =\frac{1}{r^{r/p}s^{s/q}\alpha}\cdot||\bm{x}||_p^r\cdot||\bm{y}||_q^s
  =\frac{1}{r^{r/p}s^{s/q}\alpha}.
\end{align*}
This inequality implies $\lambda_{p,q}(G)\leq r^{-r/p}s^{-s/q}\alpha^{-1}$.
   
The equality holds if $G$ is elliptic $\alpha$-normal and there is a nonzero solution 
$(\bm{x},\bm{y})$ to the following equations:
\begin{equation}\label{eq:general-solution}
rB(i_1,e)x_{i_1}^p=\cdots=rB(i_r,e)x_{i_r}^p=sB(j_1,e)y_{j_1}^q=\cdots =sB(j_s,e)y_{j_s}^q=w(e) 
\end{equation}
for any $e\in E(G)$, $T(e)=\{i_1,i_2,\ldots,i_r\}$ and $H(e)=\{j_1,j_2,\ldots,j_s\}$. Assume 
$u\in T(G)$, $v\in H(G)$, and $u\in T(e)$, $v\in H(f)$ for some arcs $e$, $f\in E(G)$. Define 
\begin{equation}\label{eq:x_u^*}
x_u^*=\bigg(\frac{w(e)}{rB(u,e)}\bigg)^{1/p}
\end{equation}
and 
\begin{equation}\label{eq:y_v^*}
y_v^*=\bigg(\frac{w(f)}{sB(v,f)}\bigg)^{1/q}.    
\end{equation}
The consistent conditions guarantee that $x_u^*$ and $y_v^*$ are independent of the choice of 
the arcs $e$ and $f$. It is easy to check that $(\bm{x}^*,\bm{y}^*)$ is a solution of 
\eqref{eq:general-solution}. Equations \eqref{eq:x_u^*} and \eqref{eq:y_v^*} also imply that
\[
\begin{cases}
rB(u,e)(x_u^*)^p=w(e), & \text{if}\ u\in T(e),\\  
sB(v,f)(y_v^*)^q=w(f), & \text{if}\ v\in H(f), 
\end{cases}   
\]
from which it follows that 
\[
||\bm{x}^*||_p^p=\sum_{u\in T(G)}(x_u^*)^p=
\sum_{e\in E(G)}\sum_{u\in T(e)}B(u,e)(x_u^*)^p=
\sum_{e\in E(G)}w(e)=1 
\]
and 
\[
||\bm{y}^*||_q^q=\sum_{v\in H(G)}(y_v^*)^q=
\sum_{f\in E(G)}\sum_{v\in H(f)}B(v,f)(y_v^*)^q=
\sum_{f\in E(G)}w(f)=1.    
\]
Therefore $\lambda_{p,q}(G)=r^{-r/p}s^{-s/q}\alpha^{-1}$, completing the proof.
\end{proof}

An $(r,s)$-directed hypergraph $G$ is called an {\em out-hyperstar} 
(or {\em in-hyperstar}) if each two arcs of $G$ share the same vertex 
in the tail (or head) of each arc. The same vertex is called the {\em center} 
of $G$.

\begin{example}
Let $G$ be an out-hyperstar with $k$ arcs and $r/p+s/q<1$. We define a weighted 
incidence matrix $B$ and $\{w(e)\}$ for $G$ as follows:
\begin{align*}
	B(v,e) & =
	\begin{cases}
		1/k, & \text{if}\ v\ \text{is the center},\\
		1, & \text{else if}\ v\in e,\\
		0, & \text{otherwise}, 
	\end{cases}\\
	w(e) & =1/k.
\end{align*}
It can be checked that $G$ is elliptic consistently $\alpha$-normal with
$\alpha=k^{(r-1)/p+s/q-1}$. Therefore \[
\lambda_{p,q}(G)=\frac{k^{1-((r-1)/p+s/q)}}{r^{r/p}s^{s/q}}.
\]
In particular, if $r/p+s/q=1$, we have 
\[
\lambda_{p,q}(G)=\frac{k^{1/p}}{r^{r/p}s^{s/q}}
\]
by taking $r/p+s/q\to 1$.
Similarly, we can prove that if $G$ is an in-hyperstar with $k$ arcs, then
\[
\lambda_{p,q}(G)=
\begin{dcases}
\frac{k^{1-(r/p+(s-1)/q)}}{r^{r/p}s^{s/q}}, & \text{if}\ r/p+s/q<1,\\
\frac{k^{1/q}}{r^{r/p}s^{s/q}}, & \text{if}\ r/p+s/q=1.	
\end{dcases}
\]\end{example}

\begin{definition}\label{def:elliptic-subnormal}
An $(r,s)$-directed hypergraph $G$ is called {\em elliptic $\alpha$-subnormal} if $r/p+s/q<1$ 
and there exist a weighted incidence matrix $B$ and $\{w(e)\}$ satisfying
\begin{enumerate}
\item[(1)] $\displaystyle\sum_{e\in E(G)}w(e)\leq 1$;

\item[(2)] $\displaystyle\sum_{e:\,u\in T(e)}B(u,e)\leq 1$,~~
$\displaystyle\sum_{e:\,v\in H(e)}B(v,e)\leq 1$,
for any $u\in T(G)$, $v\in H(G)$; 
 
\item[(3)] $\displaystyle w(e)^{1-(r/p+s/q)}\prod_{u\in T(e)}\big(B(u,e)\big)^{1/p}\cdot
\prod_{v\in H(e)}\big(B(v,e)\big)^{1/q}\geq\alpha$, for any $e\in E(G)$.
\end{enumerate}
Moreover, $G$ is called {\em strictly elliptic $\alpha$-subnormal} if it is elliptic 
$\alpha$-subnormal but not elliptic $\alpha$-normal.
\end{definition}

\begin{lemma}\label{lem:elliptic-subnormal}
Let $G$ be an $(r,s)$-directed hypergraph. If $G$ is elliptic $\alpha$-subnormal, 
then the $(p,q)$-spectral radius of $G$ satisfies
\[
\lambda_{p,q}(G)\leq\frac{1}{r^{r/p}s^{s/q}\alpha}.    
\]
\end{lemma}
       
\begin{proof}
For any nonnegative vectors $\bm{x}=(x_1,x_2,\ldots,x_m)^{\mathrm{T}}\in\mathbb{S}_{p,+}^{m-1}$
and $\bm{y}=(y_1,y_2,\ldots,y_n)^{\mathrm{T}}\in\mathbb{S}_{q,+}^{n-1}$, by 
Generalized H\"older's inequality and AM--GM inequality, we deduce that
\begin{align*}
P_G(\bm{x},\bm{y}) 
& =\sum_{e\in E(G)}\Bigg(\prod_{u\in T(e)}x_u\Bigg)\Bigg(\prod_{v\in H(e)}y_v\Bigg)\\
& \leq\frac{1}{\alpha}\sum_{e\in E(G)}\Bigg(w(e)^{1-(r/p+s/q)}\cdot\prod_{u\in T(e)}\big(B(u,e)\big)^{1/p}x_u
  \cdot\prod_{v\in H(e)}\big(B(v,e)\big)^{1/q}y_v\Bigg)\\
& \leq\frac{1}{\alpha}
    \Bigg(\sum_{e\in E(G)}\prod_{u\in T(e)}\big(B(u,e)\big)^{1/r}x_u^{p/r}\Bigg)^{r/p}
    \Bigg(\sum_{e\in E(G)}
    \prod_{v\in H(e)}\big(B(v,e)\big)^{1/s}y_v^{q/s}\Bigg)^{s/q}\\
& \leq\frac{1}{r^{r/p}s^{s/q}\alpha}\Bigg(\sum_{e\in E(G)}\sum_{u\in T(e)}B(u,e)x_u^p\Bigg)^{r/p}
    \Bigg(\sum_{e\in E(G)}\sum_{v\in H(e)}B(v,e)y_v^q\Bigg)^{s/q}\\
& \leq\frac{1}{r^{r/p}s^{s/q}\alpha}\cdot||\bm{x}||_p^r\cdot||\bm{y}||_q^s
  =\frac{1}{r^{r/p}s^{s/q}\alpha},
\end{align*}
yielding $\lambda_{p,q}(G)\leq r^{-r/p}s^{-s/q}\alpha^{-1}$. When $G$ is strictly 
elliptic $\alpha$-subnormal, this inequality is strict, and therefore 
$\lambda_{p,q}(G)<r^{-r/p}s^{-s/q}\alpha^{-1}$.
\end{proof}

\begin{definition}
An $(r,s)$-directed hypergraph $G$ is called {\em elliptic $\alpha$-supernormal} 
if $r/p+s/q<1$ and there exist a weighted incidence matrix $B$ and $\{w(e)\}$ 
satisfying
\begin{enumerate}
\item[(1)] $\displaystyle\sum_{e\in E(G)}w(e)\geq 1$;

\item[(2)] $\displaystyle\sum_{e:\,u\in T(e)}B(u,e)\geq 1$,~~
$\displaystyle\sum_{e:\,v\in H(e)}B(v,e)\geq 1$,
for any $u\in T(G)$, $v\in H(G)$; 
 
\item[(3)] $\displaystyle w(e)^{1-(r/p+s/q)}\prod_{u\in T(e)}\big(B(u,e)\big)^{1/p}\cdot
\prod_{v\in H(e)}\big(B(v,e)\big)^{1/q}\leq\alpha$,
for any $e\in E(G)$.
\end{enumerate}
Moreover, $G$ is called {\em strictly elliptic $\alpha$-supernormal} if it is 
elliptic $\alpha$-supernormal but not elliptic $\alpha$-normal.
\end{definition}
   
\begin{lemma}\label{lem:elliptic-supernormal}
Let $G$ be an $(r,s)$-directed hypergraph. If $G$ is elliptic consistently 
$\alpha$-supernormal, then the $(p,q)$-spectral radius of $G$ satisfies
\[
\lambda_{p,q}(G)\geq\frac{1}{r^{r/p}s^{s/q}\alpha}.    
\]
\end{lemma}
   
\begin{proof}
Define vectors $\bm{x}\in\mathbb{S}_{p,++}^{m-1}$ and $\bm{y}\in\mathbb{S}_{q,++}^{n-1}$ 
as follows:
\[
x_u=\bigg(\frac{w(e)}{rB(u,e)}\bigg)^{1/p},~u\in T(e);~~
y_v=\bigg(\frac{w(f)}{sB(v,f)}\bigg)^{1/q},~v\in H(f).
\]
The consistent conditions guarantee that $x_u$ and $y_v$ are independent of the 
choice of the arcs $e$ and $f$. Hence, we have 
\begin{align*}
\lambda_{p,q}(G) & \geq P_G(\bm{x},\bm{y})\\
& =
\frac{1}{r^{r/p}s^{s/q}}\sum_{e\in E(G)}
\frac{w(e)^{r/p+s/q}}{\prod\limits_{u\in T(e)}(B(u,e))^{1/p}\prod\limits_{v\in H(e)}(B(v,e))^{1/q}}\\
& \geq\frac{1}{r^{r/p}s^{s/q}\alpha}\sum_{e\in E(G)}w(e)\\
& \geq\frac{1}{r^{r/p}s^{s/q}\alpha}.
\end{align*}
When $G$ is strictly elliptic $\alpha$-supernormal, this inequality is strict, 
and therefore $\lambda_{p,q}(G)> r^{-r/p}s^{-s/q}\alpha^{-1}$.
\end{proof}

\subsection{Hyperbolic phase: $\frac{r}{p}+\frac{s}{q}>1$}

Due to the fact that the Perron--Frobenius Theorem fails for general 
$(r,s)$-directed hypergraph $G$ when $r/p+s/q>1$, the theory is less 
effective than the case $r/p+s/q\leq 1$. However, we can still
define the hyperbolic $\alpha$-normal for $r/p+s/q>1$ as
\autoref{def:elliptic normal}, and prove the following result.

\begin{theorem}
For $r/p+s/q>1$, and any $(r,s)$-directed hypergraph $G$ with $(p,q)$-spectral 
radius $\lambda_{p,q}(G)$, there exists an induced sub-dirhypergraph $G'$ of 
$\mathcal{B}(G)$ such that $G'$ is hyperbolic consistently $\alpha$-normal with 
$\alpha=(r^{r/p}s^{s/q}\lambda_{p,q}(G))^{-1}$.

Conversely, we have
\[
\lambda_{p,q}(G))=\frac{1}{r^{r/p}s^{s/q}} \max_i\big\{\alpha_i^{-1}\big\},
\]
where the maximum is taken over all $\alpha_i$ such that there is a hyperbolic 
consistent $\alpha_i$-normal labeling on some induced sub-dirhypergraph of 
$\mathcal{B}(G)$.
\end{theorem}

\begin{proof}
For short, denote $\mathcal{B}:=\mathcal{B}(G)$. Assume that 
$(\bm{x},\bm{y})\in\mathbb{S}^{m-1}_{p,+}\times\mathbb{S}^{n-1}_{q,+}$ 
is an eigenpair corresponding to $\lambda_{p,q}(G)$. Let $S_1:=\{u\in T(G): x_u>0\}$, 
$S_2:=\{v\in H(G): y_v>0\}$. Consider the induced dirhypergraph $\mathcal{B}[S_1\cup S_2]$. 
By \autoref{lem:spectral components},
\[
\lambda_{p,q}(G)=\lambda_{p,q}(\mathcal{B})=\lambda_{p,q}(\mathcal{B}[S_1\cup S_2]).
\] 
It can be proved that $G'=\mathcal{B}[S_1\cup S_2]$ is the desired induced
sub-dirhypergraph. The proof is similar to \autoref{lem:elliptic-normal}.

Conversely, assume that $G_i$ is an induced sub-dirhypergraph of $\mathcal{B}$, 
and $\{B(v,e)\}$ and $\{w(e)\}$ are hyperbolic consistent 
$\alpha_i$-normal labeling of $G_i$. Define vectors 
$\bm{x}=(x_1,\ldots,x_m)^{\mathrm{T}}\in\mathbb{S}^{m-1}_{p,++}$,
$\bm{y}=(y_1,\ldots,y_n)^{\mathrm{T}}\in\mathbb{S}^{n-1}_{q,++}$
for $G$ as follows:
\begin{align*}
x_u & =\begin{dcases}
\bigg(\frac{w(e)}{rB(u,e)}\bigg)^{1/p}, & \text{if}\ u\in T(e),\,e\in E(G_i),\\
0, & \text{otherwise},
\end{dcases}\\
y_v & =\begin{dcases}
\bigg(\frac{w(f)}{sB(v,f)}\bigg)^{1/q}, & \text{if}\ v\in H(f),\,f\in E(G_i),\\
0, & \text{otherwise}.
\end{dcases}
\end{align*}
The consistent conditions guarantee that $x_u$ and $y_v$ are independent 
of the choice of the arcs $e$ and $f$. It follows that
\begin{align*}
\lambda_{p,q}(G)\geq P_G(\bm{x},\bm{y}) 
& =\frac{1}{r^{r/p}s^{s/q}}\sum_{e\in E(G)}
\frac{w(e)^{r/p+s/q}}{\prod\limits_{u\in T(e)}(B(u,e))^{1/p}\prod\limits_{v\in H(e)}(B(v,e))^{1/q}}\\
& =\frac{1}{r^{r/p}s^{s/q}\alpha_i}\sum_{e\in E(G)}w(e)\\
& =\frac{1}{r^{r/p}s^{s/q}\alpha_i}.
\end{align*} 
Combining with the first part of this theorem, we have
\[
\lambda_{p,q}(G))=\frac{1}{r^{r/p}s^{s/q}} \max_i\big\{\alpha_i^{-1}\big\}.
\]
The proof is completed.
\end{proof}

We also can define the hyperbolic $\alpha$-subnormal for $r/p+s/q>1$ as 
\autoref{def:parabolic-subnormal}. According to the proof of 
\autoref{lem:parabolic-subnormal} and \eqref{eq:GeqHolder}, 
we still have the following result.
 
\begin{theorem}\label{thm:hyperbolic-subnormal}
Let $G$ be an $(r,s)$-directed hypergraph. If $G$ is hyperbolic 
$\alpha$-subnormal, then the $(p,q)$-spectral radius of $G$ satisfies
\[
\lambda_{p,q}(G)\leq\frac{1}{r^{r/p}s^{s/q}\alpha}.   
\]
\end{theorem}

\section{Applications}
\label{sec5}
In this section, we shall give some applications of the $\alpha$-normal 
labeling method in the study of $(p,q)$-spectral radius. For short, we 
denote $\gamma(p,q):=1-(r/p+s/q)$ in this section.

\subsection{Some degree based bounds}

\begin{proposition}
Let $G$ be an $(r,s)$-directed hypergraph with maximum out-degree $\Delta^+$
and maximum in-degree $\Delta^-$. 
\begin{enumerate}
\item[$(1)$] If $r/p+s/q\geq 1$, then
\[
\lambda_{p,q}(G)\leq\left(\frac{\Delta^+}{r}\right)^{r/p}
\left(\frac{\Delta^-}{s}\right)^{s/q}.   
\]
\item[$(2)$] If $r/p+s/q<1$, then 
\[
\lambda_{p,q}(G)\leq 
\,|G|^{1-(r/p+s/q)}\left(\frac{\Delta^+}{r}\right)^{r/p}
\left(\frac{\Delta^-}{s}\right)^{s/q}.        
\]
\end{enumerate}
\end{proposition}

\begin{proof}
(1). Assume $r/p+s/q\geq 1$. Without loss of generality, we can assume
$G$ is anadiplosis connected. Otherwise, we consider an anadiplosis 
connected component instead. Construct a weighted incidence matrix 
$B=(B(v,e))$ for $G$ as follows:
\[
B(v,e)=\begin{cases}
        1/\Delta^+, & \text{if}\ v\in T(e),\\
        1/\Delta^-, & \text{if}\ v\in H(e),\\
        0, & \text{otherwise}.
        \end{cases}
\]
For any $u\in T(G)$, $v\in H(G)$, we see that
\[
\sum_{e:\,u\in T(e)}B(u,e)\leq 1,~~
\sum_{e:\,v\in H(e)}B(v,e)\leq 1.
\]
For any arc $e\in E(G)$,
\[
\prod_{u\in T(e)}(B(u,e))^{1/p}\cdot\prod_{v\in H(e)}(B(v,e))^{1/q}\geq
\frac{1}{(\Delta^+)^{r/p}(\Delta^-)^{s/q}}.   
\]
Using \autoref{lem:parabolic-subnormal} and \autoref{thm:hyperbolic-subnormal} gives
\[
\lambda_{p,q}(G)\leq\frac{1}{r^{r/p}s^{s/q}\alpha}= 
\left(\frac{\Delta^+}{r}\right)^{r/p}\left(\frac{\Delta^-}{s}\right)^{s/q}.
\] 

(2). When $r/p+s/q<1$, we define a weighted incidence matrix $B=(B(v,e))$ and $\{w(e)\}$ 
for $G$ as follows:
\begin{align*}
B(v,e) & =\begin{cases}
        1/\Delta^+, & \text{if}\ v\in T(e),\\
        1/\Delta^-, & \text{if}\ v\in H(e),\\
        0, & \text{otherwise}.
        \end{cases}\\
w(e) & =1/|G|.
\end{align*}  
It can be checked that $G$ is elliptic $\alpha$-subnormal with 
\[
\alpha=\frac{1}{|G|^{\gamma(p,q)}(\Delta^+)^{r/p}(\Delta^-)^{s/q}}.    
\]
According to \autoref{lem:elliptic-subnormal}, we have  
\[
\lambda_{p,q}(G)\leq\frac{1}{r^{r/p}s^{s/q}\alpha}=
\,|G|^{\gamma(p,q)}\left(\frac{\Delta^+}{r}\right)^{r/p}
\left(\frac{\Delta^-}{s}\right)^{s/q}.        
\]
The proof is completed.
\end{proof}

\begin{proposition}
Let $G$ be an $(r,s)$-directed hypergraph.
\begin{enumerate}
\item[$(1)$] If $r/p+s/q\geq 1$, then
\[
\lambda_{p,q}(G)\leq\frac{1}{r^{r/p}s^{s/q}}\max_{e\in E(G)}
\Bigg\{\prod_{u\in T(e)}\big(d_u^+\big)^{1/p}\prod_{v\in H(e)}\big(d_v^-\big)^{1/q}\Bigg\}.  
\]
\item[$(2)$] If $r/p+s/q<1$, then
\[
\lambda_{p,q}(G)\leq\frac{|G|^{1-(r/p+s/q)}}{r^{r/p}s^{s/q}}\max_{e\in E(G)}
\Bigg\{\prod_{u\in T(e)}\big(d_u^+\big)^{1/p}\prod_{v\in H(e)}\big(d_v^-\big)^{1/q}\Bigg\}.  
\] 
\end{enumerate}
\end{proposition}
            
\begin{proof}
(1). Assume $r/p+s/q\geq 1$. Without loss of generality, we can assume
$G$ is anadiplosis connected. Otherwise, we consider an anadiplosis 
connected component instead. We construct a weighted incidence matrix 
$B=(B(v,e))$ for $G$ as follows:
\[
B(v,e)=\begin{cases}
    1/d_v^+, & \text{if}\ v\in T(e),\\
    1/d_v^-, & \text{if}\ v\in H(e),\\
    0, & \text{otherwise}.
    \end{cases}
\]
Clearly, for any $u\in T(G)$, $v\in H(G)$, we see that
\[
\sum_{e:\,u\in T(e)}B(u,e)=\sum_{e:\,v\in H(e)}B(v,e)=1.    
\]
For any arc $e\in E(G)$, 
\[
\prod_{u\in T(e)}(B(u,e))^{1/p}\cdot\prod_{v\in H(e)}(B(v,e))^{1/q}\geq    
\frac{1}{\max\limits_{e\in E(G)}\bigg\{\prod\limits_{u\in T(e)}(d_u^+)^{1/p}
\prod\limits_{v\in H(e)}\big(d_v^-\big)^{1/q}\bigg\}}.
\]
By \autoref{lem:parabolic-subnormal} and \autoref{thm:hyperbolic-subnormal}, we have 
\[
\lambda_{p,q}(G)\leq\frac{1}{r^{r/p}s^{s/q}\alpha}=\frac{1}{r^{r/p}s^{s/q}}
\max_{e\in E(G)}\Bigg\{\prod_{u\in T(e)}\big(d_u^+\big)^{1/p}\prod_{v\in H(e)}\big(d_v^-\big)^{1/q}\Bigg\}.    
\]

(2). Assume $r/p+s/q<1$, we define a weighted incidence matrix $B=(B(v,e))$ and $\{w(e)\}$ 
for $G$ as follows:
\begin{align*}
B(v,e) & =\begin{cases}
    1/d_v^+, & \text{if}\ v\in T(e),\\
    1/d_v^-, & \text{if}\ v\in H(e),\\
    0, & \text{otherwise}.
    \end{cases}\\
w(e) & =1/|G|.
\end{align*}
It can be checked that $G$ is elliptic $\alpha$-subnormal with
\[
\alpha=\frac{1}{|G|^{\gamma(p,q)}}\max\limits_{e\in E(G)}
\bigg\{\prod\limits_{u\in T(e)}(d_u^+)^{1/p}\prod\limits_{v\in H(e)}\big(d_v^-\big)^{1/q}\bigg\}.  
\]
By \autoref{lem:elliptic-subnormal} we have
\[
\lambda_{p,q}(G)\leq\frac{1}{r^{r/p}s^{s/q}\alpha}
=\frac{|G|^{\gamma(p,q)}}{r^{r/p}s^{s/q}}\max_{e\in E(G)}
\Bigg\{\prod_{u\in T(e)}\big(d_u^+\big)^{1/p}\prod_{v\in H(e)}\big(d_v^-\big)^{1/q}\Bigg\}.  
\]  
The proof is completed. 
\end{proof}

\subsection{Monotonicity and convexity of $\lambda_{p,q}(G)$}

In this subsection, we consider $\lambda_{p,q}(G)$ as a function of $p$, $q$ for 
a fixed $(r,s)$-directed hypergraph $G$, and study some properties of the function 
$\lambda_{p,q}(G)$.

\begin{theorem}
Let $G$ be an $(r,s)$-directed hypergraph with $r/p+s/q<1$. Then the function 
$(r|G|)^{r/p}(s|G|)^{s/q}\lambda_{p,q}(G)$ is non-increasing in both $p$ and $q$.
\end{theorem}

\begin{proof}
Assume that $G$ is elliptic consistently $\alpha$-normal with weighted incidence 
matrix $B$ and weights $\{w(e)\}$ for $\lambda_{p,q}(G)$. Let $p<p'$. We define 
a weighted incidence matrix $B'$ and $\{w'(e)\}$ for $\lambda_{p',q}(G)$ as follows:
\begin{align*}
B'(v,e) & =
\begin{dcases}
(B(v,e))^{p'/p}, & \text{if}\ v\in T(e),\\
B(v,e), & \text{if}\ v\in H(e),
\end{dcases}\\
w'(e) & =\frac{w(e)^{\gamma(p,q)/\gamma(p',q)}}{|G|^{\frac{r/p-r/p'}{\gamma(p',q)}}}.
\end{align*}
In what follows, we shall prove that $\{B'(v,e)\}$ and $\{w'(e)\}$ are
elliptic $\alpha'$-subnormal labeling for $\lambda_{p',q}(G)$ with 
$\alpha'=\alpha |G|^{r/p'-r/p}$.

(i). Using H\"older's inequality gives
\[
\sum_{e\in E(G)}w'(e)=|G|^{-\frac{r/p-r/p'}{\gamma(p',q)}}
\sum_{e\in E(G)} w(e)^{\gamma(p,q)/\gamma(p',q)}\leq 1.
\]

(ii). For any $u\in T(G)$ and $v\in H(G)$, we have
\[
\sum_{e:\,u\in T(e)}B'(u,e)=\sum_{e:\,u\in T(e)}(B(u,e))^{p'/p}\leq
\sum_{e:\,u\in T(e)}B(u,e)=1
\]
and 
\[
\sum_{e:\,v\in H(e)}B'(v,e)=\sum_{e:\,v\in H(e)}B(v,e)=1.
\]

(iii). For each arc $e$, we have
\begin{align*}
& w'(e)^{\gamma(p',q)}\prod_{u\in T(e)}(B'(u,e))^{1/p'}
\cdot \prod_{v\in H(e)}(B'(v,e))^{1/q}\\
= &~\frac{w(e)^{\gamma(p,q)}}{|G|^{r/p-r/p'}}
\prod_{u\in T(e)}(B(u,e))^{1/p}
\cdot \prod_{v\in H(e)}(B(v,e))^{1/q}\\
= &~\frac{\alpha}{|G|^{r/p-r/p'}}. 
\end{align*}
Hence, $G$ is elliptic $\alpha'$-subnormal for $\lambda_{p',q}(G)$ with 
$\alpha'=\alpha |G|^{r/p'-r/p}$. It follows from \autoref{lem:elliptic-subnormal}
that 
\[
\lambda_{p',q}(G)\leq\frac{1}{r^{r/p'}s^{s/q}\alpha'}=
(r|G|)^{r/p-r/p'}\lambda_{p,q}.
\]
Therefore, we obtain
\[
(r|G|)^{r/p'}\lambda_{p',q}(G)\leq (r|G|)^{r/p}\lambda_{p,q}(G).
\]
Similarly, for $q'>q$, we can prove that
\[
(s|G|)^{s/q'}\lambda_{p,q'}(G)\leq (s|G|)^{s/q}\lambda_{p,q}(G).
\]
Thus, for any $p'>q$ and $q'>q$, we have
$$(r|G|)^{r/p'}(s|G|)^{s/q'}\lambda_{p',q'}(G)
\leq (r|G|)^{r/p'}(s|G|)^{s/q}\lambda_{p',q}(G)
\leq  (r|G|)^{r/p}(s|G|)^{s/q}\lambda_{p,q}(G).$$
The proof is completed. 
\end{proof}

\begin{lemma}\label{lem:w(e) bound}
Let $G$ be an $(r,s)$-directed hypergraph with $r/p+s/q<1$.
Suppose that $G$ is elliptic consistently $\alpha$-normal with 
weights $\{w(e)\}$. Then
\[
[\alpha (\delta^+)^{r/p}(\delta^-)^{s/q}]^{1/[1-(r/p+s/q)]}\leq
w(e)\leq [\alpha (\Delta^+)^{r/p}(\Delta^-)^{s/q}]^{1/[1-(r/p+s/q)]}.
\]
\end{lemma}

\begin{proof}
The consistent conditions in \autoref{def:elliptic normal} imply that
\[
B(u,e)=
\begin{dcases}
\frac{w(e)}{\sum_{f:\,u\in T(f)}w(f)}, & \text{if}\ u\in T(e),\\
\frac{w(e)}{\sum_{f:\,u\in H(f)}w(f)}, & \text{if}\ u\in H(e).
\end{dcases}
\]
By item (3) in \autoref{def:elliptic normal}, we obtain
\begin{equation}\label{eq:equality w(e)}
w(e)=\alpha\prod_{u\in T(e)}
\Bigg(\sum_{f:\,u\in T(f)}w(f)\Bigg)^{1/p}\cdot
\prod_{v\in H(e)}
\Bigg(\sum_{f:\,v\in H(f)}w(f)\Bigg)^{1/q}
\end{equation}
Without loss of generality, assume $w(e_1)=\min\{w(e): e\in E(G)\}$, and $w(e_2)=\max\{w(e): e\in E(G)\}$.
Using equation \eqref{eq:equality w(e)} gives 
\begin{align*}
w(e_1) & \geq\alpha [(\delta^+)w(e_1)]^{r/p}\cdot [(\delta^{-})w(e_1)]^{s/q}\\
& =\alpha(\delta^+)^{r/p}(\delta^-)^{s/q}\cdot w(e_1)^{r/p+s/q},
\end{align*}
which follows that
\[
w(e_1)\geq[\alpha (\delta^+)^{r/p}(\delta^-)^{s/q}]^{1/\gamma(p,q)}.
\]
Similarly, we can prove the right side.
\end{proof}

\begin{theorem}
Suppose that $G$ is an $(r,s)$-directed hypergraph with $r/p+s/q<1$. Let
\[
f_G(x):=\bigg(\Big(\frac{r}{\Delta^+}\Big)^{r/(px)}
\Big(\frac{s}{\Delta^-}\Big)^{s/(qx)}
\lambda_{px,qx}(G)\bigg)^{\frac{1}{1-(r/(px)+s/(qx))}},
\]
and
\[
g_G(x):=\bigg(\Big(\frac{r}{\delta^+}\Big)^{r/(px)}
\Big(\frac{s}{\delta^-}\Big)^{s/(qx)}
\lambda_{px,qx}(G)\bigg)^{\frac{1}{1-(r/(px)+s/(qx))}},
\]
then $f_G(x)$ is non-decreasing on $(r/p+s/q,\infty)$ while $g_G(x)$ 
is non-increasing on $(r/p+s/q,\infty)$.
\end{theorem}

\begin{proof}
For any $x_1>r/p+s/q$, let $G$ be elliptic consistently $\alpha_1$-normal 
with weighted incidence matrix $B_1$ and weights $\{w_1(e)\}$ for 
$\lambda_{px_1,qx_1}(G)$. Therefore
\[
\begin{dcases}
\sum_{e\in E(G)}w_1(e)=1,\\
\sum_{e:\,u\in T(e)}B_1(u,e)=\sum_{e:\,v\in H(e)}B_1(v,e)=1,~
u\in T(G),\,v\in H(G),\\
w_1(e)^{\gamma(px_1,qx_1)}\prod_{u\in T(e)}\big(B_1(u,e)\big)^{1/(px_1)}\cdot
\prod_{v\in H(e)}\big(B_1(v,e)\big)^{1/(qx_1)}=\alpha_1.
\end{dcases}
\]
Let $x_2>x_1$. We now define a weighted incidence matrix $B_2$ and $\{w_2(e)\}$
for $\lambda_{px_2,qx_2}(G)$ as follows:
\[
B_2(v,e)=B_1(v,e),~w_2(e)=w_1(e). 
\]
It is clear that
\[
\sum_{e\in E(G)}w_2(e)=\sum_{e\in E(G)}w_1(e)=1.
\]
We also have
\[
\sum_{e:\,u\in T(e)}B_2(u,e)=1,~~
\sum_{e:\,v\in H(e)}B_2(v,e)=1.
\]
Using \autoref{lem:elliptic-normal} gives
\begin{align*}
  &~w_2(e)^{\gamma(px_2,qx_2)}\prod_{u\in T(e)}(B_2(u,e))^{1/(px_2)}\cdot%
  \prod_{v\in H(e)}(B_2(v,e))^{1/(qx_2)}\\
= &~w_1(e)^{\gamma(px_2,qx_2)}\prod_{u\in T(e)}(B_1(u,e))^{1/(px_2)}\cdot%
  \prod_{v\in H(e)}(B_1(v,e))^{1/(qx_2)}\\
= &~w_1(e)^{1-x_1/x_2}\cdot\alpha_1^{x_1/x_2}\\
\leq &~\alpha_1^{\gamma(px_2,qx_2)/\gamma(px_1,qx_1)}
\left((\Delta^+)^{r/(px_1)}(\Delta^-)^{s/(qx_1)}\right)^{\frac{1-x_1/x_2}{\gamma(px_1,qx_1)}}.
\end{align*}
Therefore, $G$ is elliptic consistently $\alpha_2$-supernormal for 
$\lambda_{px_2,qx_2}(G)$ with 
\[
\alpha_2=\alpha_1^{\gamma(px_2,qx_2)/\gamma(px_1,qx_1)}
\left((\Delta^+)^{r/(px_1)}(\Delta^-)^{s/(qx_1)}\right)^{\frac{1-x_1/x_2}{\gamma(px_1,qx_1)}}.
\]
According to \autoref{lem:elliptic-supernormal} and 
$(\alpha_1)^{-1}=r^{r/(px_1)}s^{s/(qx_1)}\lambda_{px_1,qx_1}(G)$, 
we see
\begin{align*}
\lambda_{px_2,qx_2}(G)
& \geq\frac{1}{r^{r/(px_2)}s^{s/(qx_2)}\alpha_2}\\
& =\frac{\big[r^{r/(px_1)}s^{s/(qx_1)}\lambda_{px_1,qx_1}(G)\big]^{\frac{\gamma(px_2,qx_2)}{\gamma(px_1,qx_1)}}}%
{r^{r/(px_2)}s^{s/(qx_2)}\big[(\Delta^+)^{r/(px_1)}(\Delta^-)^{s/(qx_1)}\big]^{\frac{1-x_1/x_2}{\gamma(px_1,qx_1)}}}\\
& =\bigg[\Big(\frac{r}{\Delta^+}\Big)^{r/(px_1)}\Big(\frac{s}{\Delta^-}\Big)^{s/(qx_1)}%
   \bigg]^{\frac{\gamma(px_2,qx_2)}{\gamma(px_1,qx_1)}}
\Big(\frac{\Delta^+}{r}\Big)^{r/(px_2)}\Big(\frac{\Delta^-}{s}\Big)^{s/(qx_2)},
\end{align*}
which implies that $f_G(x)$ is non-decreasing in $x$. Similarly, 
we can prove that $g_G(x)$ is non-increasing on $(r/p+s/q,\infty)$.
\end{proof}

\begin{theorem}
For any $(r,s)$-directed hypergraph $G$ with $r/p+s/q<1$, the function 
$pq\log\,(\lambda_{p,q}(G))$ is concave upward in $p$ $($and in $q$$)$.
\end{theorem}

\begin{proof}
For any $p_1<p<p_2$, write $p=\mu p_1+(1-\mu)p_2$, where $\mu=(p_2-p)/(p_2-p_1)$.
Let $G$ be elliptic consistently $\alpha_i$-normal with weighted incident matrix 
$B_i$ and $\{w_i(e)\}$ for $\lambda_{p_i,q}(G)$, $i=1$, $2$. 

We define a weighted incidence matrix $B$ and $\{w(e)\}$ for $\lambda_{p,q}(G)$ 
as follows:
\begin{align*}
B(u,e) & =\mu B_1(u,e)+(1-\mu)B_2(u,e),~\text{if}\ u\in T(e),\\
B(v,e) & =\eta B_1(v,e)+(1-\eta) B_2(v,e),~\text{if}\ v\in H(e),\\
w(e) & =\xi w_1(e)+(1-\xi)w_2(e),
\end{align*}
where
\[
\eta=\frac{p_1}{p}\mu,~\xi=\frac{p_1q-(rq+sp_1)}{pq-(rq+sp)}\mu.  
\]
For any vertices $u\in T(G)$, $v\in H(G)$, we have
\begin{align*}
\sum_{e:\,u\in T(e)}B(u,e) & =\mu\sum_{e:\,u\in T(e)}B_1(u,e)
+(1-\mu)\sum_{e:\,u\in T(e)}B_2(u,e)\\
& =\mu+(1-\mu)=1.
\end{align*}
Also, we have
\begin{align*}
\sum_{e:\,v\in H(e)}B(v,e) & =\eta\sum_{e:\,v\in H(e)}B_1(u,e)
+(1-\eta)\sum_{e:\,v\in H(e)}B_2(v,e)\\
& =\eta+(1-\eta)=1.
\end{align*}
For each arc $e\in E(G)$, it follows from Young's inequality that
\begin{align*}
&~\Bigg(w(e)^{1-(r/p+s/q)}\prod_{u\in T(e)}(B(u,e))^{1/p}\prod_{v\in H(e)}(B(v,e))^{1/q}\Bigg)^{pq}\\
\geq &~w_1(e)^{\xi[pq-(rq+sp)]}\prod_{u\in T(e)}(B_1(u,e))^{\mu q}
\prod_{v\in H(e)}(B_1(v,e))^{\eta p}\\
& \times w_2(e)^{(1-\xi)[pq-(rq+sp)]}\prod_{u\in T(e)}(B_2(u,e))^{(1-\mu) q}
\prod_{v\in H(e)}(B_2(v,e))^{(1-\eta) p}\\
= &~(\alpha_1)^{p_1q\mu}(\alpha_2)^{p_2q(1-\mu)}.  
\end{align*}
Hence, $G$ is elliptic $\alpha$-subnormal for $\lambda_{p,q}(G)$ with
$\alpha^{pq}=(\alpha_1)^{p_1q\mu}(\alpha_2)^{p_2q(1-\mu)}$. Using 
\autoref{lem:elliptic-subnormal} gives
\begin{align*}
pq\log\,(\lambda_{p,q}(G)) & \leq -\log\,(r^{rq}s^{sp}\alpha^{pq})\\
& =-\log\,\big(r^{rq}s^{sp}\alpha_1^{p_1q\mu}\alpha_2^{p_2q(1-\mu)}\big)\\
& =-\log\,(r^{rq}s^{sp})-p_1q\mu\log \alpha_1-p_2q(1-\mu)\log\alpha_2\\
& =\mu p_1q\log\,(\lambda_{p_1,q}(G))+(1-\mu)p_2q\log\,(\lambda_{p_2,q}(G)),
\end{align*}
which implies that the function $pq\log\,(\lambda_{p,q}(G))$ is concave 
upward in $p$. Similarly, we can prove that $pq\log\,(\lambda_{p,q}(G))$ 
is also concave upward in $q$.
\end{proof}

\begin{theorem}
For any $(r,s)$-directed hypergraph $G$ with $r/p+s/q<1$, the function
\[
h_G(1/p,1/q):=\log\big(\lambda_{p,q}(G)\big)
\]
is concave upward in $1/p$ and $1/q$.
\end{theorem}

\begin{proof}
According to \autoref{lem:elliptic-normal}, let $G$ be elliptic consistently 
$\alpha_i$-normal with weighted incident matrix $B_i$ and $\{w_i(e)\}$ for
$\lambda_{p_i,q_i}(G)$, where $(\alpha_i)^{-1}=r^{r/p_i}s^{s/q_i}\lambda_{p_i,q_i}(G)$,
$i=1$, $2$.

For any $(1/p,1/q)$, write 
\[
\bigg(\frac{1}{p},\frac{1}{q}\bigg)=
\mu\bigg(\frac{1}{p_1},\frac{1}{q_1}\bigg)+
(1-\mu)\bigg(\frac{1}{p_2},\frac{1}{q_2}\bigg),
\]
where 
\[
\mu=\frac{p_1(p_2-p)}{p(p_2-p_1)}=
\frac{q_1(q_2-q)}{q(q_2-q_1)}.
\]
Furthermore, let 
\[
\mu_1=\frac{p}{p_1}\mu,~\mu_2=\frac{q}{q_1}\mu,~
\xi=\frac{\gamma(p_1,q_1)}{\gamma(p,q)}\mu.
\]
We define a weighted incidence matrix $B$ and $\{w(e)\}$ for
$\lambda_{p,q}(G)$ as follows:
\begin{align*}
B(v,e) & =
\begin{cases}
\mu_1B_1(v,e)+(1-\mu_1)B_2(v,e), & \text{if}\ v\in T(e),\\
\mu_2B_1(v,e)+(1-\mu_2)B_2(v,e), & \text{if}\ v\in H(e),
\end{cases}\\
w(e) & =\xi w_1(e)+(1-\xi)w_2(e).
\end{align*}
%Therefore we have
%\[
%\sum_{e\in E(G)}w(e)=\xi\sum_{e\in E(G)}w_1(e)+(1-\xi)\sum_{e\in E(G)}
%w_2(e)=\xi+(1-\xi)=1.
%\]
%For any vertices $u\in T(G)$, $v\in H(G)$, we have
%\begin{align*}
%\sum_{e:\,u\in T(e)}B(u,e) & =\mu_1\sum_{e:\,u\in T(e)}B_1(u,e)
%+(1-\mu_1)\sum_{e:\,u\in T(e)}B_2(u,e)\\
%& =\mu_1+(1-\mu_1)=1,
%\end{align*}
%and
%\begin{align*}
%\sum_{e:\,v\in H(e)}B(v,e) & =\mu_2\sum_{e:\,v\in H(e)}B_1(u,e)
%+(1-\mu_2)\sum_{e:\,v\in H(e)}B_2(v,e)\\
%& =\mu_2+(1-\mu_2)=1.
%\end{align*}
%For each arc $e\in E(G)$, it follows from Young's inequality that
%\begin{align*}
%&~w(e)^{\gamma(p,q)}\prod_{u\in T(e)}(B(u,e))^{1/p}\cdot
%\prod_{v\in T(e)}(B(v,e))^{1/q}\\
%\geq &~w_1(e)^{\xi\gamma(p,q)}
%\prod_{u\in T(e)}(B_1(u,e))^{\mu_1/p}\cdot
%\prod_{v\in T(e)}(B_1(v,e))^{\mu_2/q}\cdot\\
%&~w_2(e)^{(1-\xi)\gamma(p,q)}
%\prod_{u\in T(e)}(B_2(u,e))^{(1-\mu_1)/p}\cdot
%\prod_{v\in T(e)}(B_2(v,e))^{(1-\mu_2)/q}\\
%= &~(\alpha_1)^{\mu}(\alpha_2)^{1-\mu}. 
%\end{align*}
It can be checked that $G$ is elliptic $\alpha$-subnormal for $\lambda_{p,q}(G)$ with 
$\alpha=(\alpha_1)^{\mu}(\alpha_2)^{1-\mu}$. By \autoref{lem:elliptic-subnormal},
we have
\begin{align*}
\log\big(\lambda_{p,q}(G)\big)
=    &~\log\big(r^{r/p}s^{s/q}\lambda_{p,q}(G)\big) -\frac{r\log r}{p} -\frac{s\log s}{q}\\
\leq &~-\log\alpha -\frac{r\log r}{p} -\frac{s\log s}{q}\\
=    &~-(\mu\log\alpha_1+(1-\mu)\log\alpha_2)-\frac{r\log r}{p} -\frac{s\log s}{q}\\
=    &~\mu\log\big(r^{r/p_1}s^{s/q_1}\lambda_{p_1,q_1}(G)\big)+
      (1-\mu)\log\big(r^{r/p_2}s^{s/q_2}\lambda_{p_2,q_2}(G)\big)\\
    &~-\frac{r\log r}{p} -\frac{s\log s}{q}\\
=   &~\mu\log\big(\lambda_{p_1,q_1}(G)\big)+
     (1-\mu)\log\big(\lambda_{p_2,q_2}(G)\big).   
\end{align*}
Thus the function $h_G(1/p,1/q)$ is concave upward in $1/p$ and $1/q$.
\end{proof}

\begin{corollary}
For any $(r,s)$-directed hypergraph $G$ with $r/p+s/q<1$, the function
$\log\big(\lambda_{px,qx}(G)\big)$ is concave upward in $1/x$ on the interval $(r/p+s/q,\infty)$.
\end{corollary}

\begin{theorem}
Let $G$ be an $(r,s)$-directed hypergraph and $r/p+s/q<1$.
Then the function $x\log\,(\lambda_{px,qx}(G))$ is concave upward
in $x$ on the interval $(r/p+s/q,\infty)$.
\end{theorem}

\begin{proof}
For any $x_2>x_1>r/p+s/q$, let $G$ be elliptic consistently 
$\alpha_i$-normal with weighted incident matrix $B_i$ and $\{w_i(e)\}$ 
for $\lambda_{px_i,qx_i}(G)$, $i=1$, $2$.

For $x>r/p+s/q$, write $x=\mu x_1+(1-\mu)x_2$, where $\mu=(x_2-x)/(x_2-x_1)$.
We define a weighted incidence matrix $B$ and $\{w(e)\}$ for 
$\lambda_{px,qx}(G)$ as follows:
\[
B(v,e)=\mu B_1(v,e)+(1-\mu)B_2(v,e),~
w(e)=\xi w_1(e)+(1-\xi)w_2(e),
\]
where 
\[
\xi=\frac{\mu x_1\gamma(px_1,qx_1)}{x-(r/p+s/q)}.
\]
By some simple computation, we have
\begin{align*}
&~\bigg[w(e)^{\gamma(px,qx)}\prod_{u\in T(e)}(B(u,e))^{1/px}\prod_{v\in H(e)}(B(v,e))^{1/qx}\bigg]^{x}\\
= &~w(e)^{x-(r/p+s/q)} \prod_{u\in T(e)}(B(u,e))^{1/p}\prod_{v\in H(e)}(B(v,e))^{1/q}\\
\geq &~w_1(e)^{\xi[x-(r/p+s/q)]} \prod_{u\in T(e)}(B_1(u,e))^{\mu/p}\prod_{v\in H(e)}(B_1(v,e))^{\mu/q}\\
& \times w_2(e)^{(1-\xi)[x-(r/p+s/q)]} \prod_{u\in T(e)}(B_2(u,e))^{(1-\mu)/p}\prod_{v\in H(e)}(B_2(v,e))^{(1-\mu)/q}\\
= &~(\alpha_1)^{\mu x_1}(\alpha_2)^{(1-\mu)x_2}.
\end{align*}
Hence, $G$ is elliptic $\alpha$-subnormal for $\lambda_{px,qx}(G)$ with
$\alpha^x=(\alpha_1)^{\mu x_1}(\alpha_2)^{(1-\mu)x_2}$. It follows from
\autoref{lem:elliptic-subnormal} that
\[
x\log\,(\lambda_{px,qx}(G))\leq\mu x_1\log\,(\lambda_{px_1,qx_1}(G))
+(1-\mu)x_2\log\,(\lambda_{px_2,qx_2}(G)).
\]
The proof is completed.
\end{proof}

\subsection{Miscellaneous results}

The following theorem establish an relation of spectral radius between $G$ and
the underlying of $\mathcal{B}(G)$.

\begin{theorem}
Let $G$ be an $(r,s)$-directed hypergraph with $r/p+s/q=1$.
Suppose that $\overline{G}$ is the underlying hypergraph of $\mathcal{B}(G)$,
and $\rho(\overline{G})$ is the spectral radius of $\overline{G}$. 
\begin{enumerate}
\item[$(1)$] If $p\leq q$, then 
\[
\lambda_{p,q}(G)\leq\frac{1}{r^{r/p}s^{s/q}}(\rho(\overline{G}))^{(r+s)/p};    
\]
\item[$(2)$] If $p>q$, then
\[
\lambda_{p,q}(G)\leq\frac{1}{r^{r/p}s^{s/q}}(\rho(\overline{G}))^{(r+s)/q};    
\]
\item[$(3)$] If $p=q=r+s$, then 
\[
\lambda_{p,q}(G)=\frac{1}{\sqrt[r+s]{r^rs^s}}\rho(\overline{G}).    
\]
\end{enumerate}
\end{theorem}

\begin{proof}
By \autoref{lem:spectral components}, we may assume that $G$ is anadiplosis connected.
According to \autoref{thm:uniform hypergraph normal label}, let 
$\overline{B}=(\overline{B}(v,\overline{e}))$ be the weighted incidence matrix of 
$\overline{G}$ satisfying\\
(i) $\sum_{\overline{e}:\, v\in\overline{e}}\overline{B}(v,\overline{e})=1$, 
for any $v\in V(\overline{G})$;\\
(ii) $\prod_{v:\,v\in\overline{e}}\overline{B}(v,\overline{e})=\overline{\alpha}
=(\rho(\overline{G}))^{-(r+s)}$, for any $\overline{e}\in E(\overline{G})$;\\
(iii) $\prod_{i=1}^{\ell}\frac{B(v_{i-1},e_i)}{B(v_i,e_i)}=1$, for any cycle
$v_0e_1v_1e_2\cdots v_{\ell-1}e_{\ell}(v_{\ell}=v_0)$.\\
Now we define a weighted incidence matrix $B=(B(v,e))$ for $G$ as 
\begin{equation}\label{eq:label of underlying}
B(v,e)=\overline{B}(v,\overline{e}),\ \text{for any}\ e\in E(G).
\end{equation}
Clearly, for any $u\in T(G)$, $v\in H(G)$,
\[
\sum_{e:\,u\in T(e)}B(u,e)=
\sum_{\overline{e}:\,u\in \overline{e}}\overline{B}(u,\overline{e})=1,~
\sum_{e:\,v\in H(e)}B(v,e)=
\sum_{\overline{e}:\,v\in\overline{e}}\overline{B}(v,\overline{e})=1.        
\]
Also, for any $e\in E(G)$,
\[
\prod_{u\in T(e)}\big(B(u,e)\big)^{1/p}\cdot
\prod_{v\in H(e)}\big(B(v,e)\big)^{1/q}\geq
\begin{cases}
(\overline{\alpha})^{1/p}, & \text{if}\ p\leq q,\\
(\overline{\alpha})^{1/q}, & \text{if}\ p>q.
\end{cases}
\]
If $p\leq q$, $G$ is parabolic $(\overline{\alpha})^{1/p}$-subnormal.
By \autoref{lem:parabolic-subnormal} we have 
\[
\lambda_{p,q}(G)\leq\frac{1}{r^{r/p}s^{s/q}(\overline{\alpha})^{1/p}}
=\frac{1}{r^{r/p}s^{s/q}}(\rho(\overline{G}))^{(r+s)/p}.
\]
If $p>q$, $G$ is parabolic $(\overline{\alpha})^{1/q}$-subnormal.
By \autoref{lem:parabolic-subnormal} we conclude that 
\[
\lambda_{p,q}(G)\leq\frac{1}{r^{r/p}s^{s/q}(\overline{\alpha})^{1/q}}
=\frac{1}{r^{r/p}s^{s/q}}(\rho(\overline{G}))^{(r+s)/q}.
\]
Let $p=q=r+s$. By \autoref{def:parabolic-normal}, equation 
\eqref{eq:label of underlying} is a parabolic consistent 
$(\overline{\alpha})^{1/(r+s)}$-normal labeling of $G$. 
Therefore $\lambda_{p,q}(G)=(r^rs^s)^{-1/(r+s)}\rho(\overline{G})$.
\end{proof}

Let $G=(V,E)$ be an $(r,s)$-directed hyergraph. For each $u\in V$ (and $e\in E$), 
let $V_u$ (and $T_e$, $H_e$) be a new vertex set with $k$ (and $a$, $b$) elements 
such that all these new sets are pairwise disjoint. Then the {\em power} of $G$, 
denoted by $G(k;a,b)$, is defined as the $(kr+a,ks+b)$-directed hypergraph with 
the vertex set
\[
V(G(k;a,b))=\Bigg(\bigcup_{u\in V}V_u\Bigg)\bigcup
\Bigg(\bigcup_{e\in E}(T_e\cup H_e)\Bigg)    
\]
and arc set
\[
E(G(k;a,b))=\Bigg\{\widetilde{e}=\Bigg(\bigcup_{u\in T(e)}\big(V_u\bigcup T_e\big),~
\bigcup_{u\in H(e)}\big(V_u\bigcup H_e\big)\Bigg): e\in E(G)\Bigg\}.    
\]
    
\begin{theorem}
Let $G$ be an $(r,s)$-directed hyergraph, and $G(k;a,b)$ be the power of $G$ with 
$as=br$. Then
\[
\rho(G(k;a,b))=\frac{(\sqrt{rs}\,\rho(G))^{kr/(kr+a)}}{\sqrt{(kr+a)(ks+b)}}.    
\]
\end{theorem}

\begin{proof}
Assume that $B=(B(u,e))$ is the parabolic consistent $\alpha$-normal labeling
of $G$. Now define a weighted incidence matrix $B'$ for $G(k;a,b)$ as follows:
\[
B'(v,\widetilde{e})=
\begin{cases}
    B(u,e), & \text{if}\ v\in V_u\ \text{for some}\ u\in e,\\
    1,      & \text{if}\ v\in T_e\cup H_e,\\
    0,      & \text{otherwise}.
\end{cases}    
\]
Clearly, for any vertex $v\in V(G(k;a,b))$,
\[
\sum_{\widetilde{e}:\,v\in T(\widetilde{e})}B'(v,\widetilde{e})
=\sum_{\widetilde{e}:\,v\in H(\widetilde{e})}B(v,\widetilde{e})=1.    
\]
Notice that $as=br$. Therefore, $G(k;a,b)$ is parabolic consistently 
$\alpha'$-normal with
\[
\alpha'=\prod_{v\in T(\widetilde{e})}(B'(v,\widetilde{e}))^{1/(kr+a)}\cdot
\prod_{v\in H(\widetilde{e})}(B'(v,\widetilde{e}))^{1/(ks+b)}=\alpha^{kr/(kr+a)}.
\]
It follows from \autoref{lem:iff} that
\begin{align*}
\rho(G(k;a,b)) & =\frac{1}{\sqrt{(kr+a)(ks+b)}\,\alpha'}\\
& =\frac{1}{\sqrt{(kr+a)(ks+b)}\,\alpha^{kr/(kr+a)}}\\
& =\frac{(\sqrt{rs}\,\rho(G))^{kr/(kr+a)}}{\sqrt{(kr+a)(ks+b)}}.    
\end{align*}
The proof is completed.
\end{proof}

\section{Concluding remarks}
\label{sec6}
In this paper, we establish an initial spectral theory of directed hypergraphs
by introducing the $(p,q)$-spectral radius $\lambda_{p,q}(G)$ for an $(r,s)$-directed 
hypergraph $G$. More precisely, we present some properties of $\lambda_{p,q}(G)$, and
develop a simple method for calculating $\lambda_{p,q}(G)$ via weighted incident matrix, 
as well as for comparing the $\lambda_{p,q}(G)$ with a particular value. The main 
results of this paper are focus on general $p$, $q\geq 1$. It is interesting to 
consider the case $p=2r$, $q=2s$, in which case the statements are concise and 
nontrivial. That would be our next topic to investigate.

For directed graphs, it is known that there are several different matrices associated 
to a directed graph $G$ to capture the adjacency of the directed graph. One candidate 
is the adjacency matrix $A(G)$, which is not symmetric. The $(i,j)$-entry of $A(G)$ is 
$1$ if there is an arc from the vertex $i$ to $j$, and $0$ otherwise (see more in 
\cite{Brualdi2010}). Another candidate is the skew-symmetric adjacency matrix, where 
the $(i,j)$-entry is $1$ if there is an arc from $i$ to $j$, and $-1$ if there is an 
arc from $j$ to $i$ (and $0$ otherwise) \cite{Cavers2012}. Recently, the Hermitian 
adjacency matrix $H(G)$ is introduced by Guo and Mohar \cite{GuoMohar2016}, and 
independently by Liu and Li \cite{LiuLi2015}. The $(i,j)$-entry $h_{ij}$ of $H(G)$ 
is given by 
\[
h_{ij}=\begin{cases}
1, & \text{if}\ （i,j)\in E(G)\ \text{and}\ (j,i)\in E(G),\\
\mathbf{i}, & \text{if}\ (i,j)\in E(G)\ \text{and}\ (j,i)\notin E(G),\\
-\mathbf{i}, & \text{if}\ (i,j)\notin E(G)\ \text{and}\ (j,i)\in E(G),\\
0, & \text{otherwise},
\end{cases}   
\]
where $\mathbf{i}$ is the imaginary unit. This paper provides a new direction to study 
the spectral properties of directed graphs, which have a great relationship with the 
anadiplosis connectedness of directed graphs. It would be an interesting topic to study 
the spectrum of a directed graph via the singular values of its adjacency matrix $A$ in 
\autoref{def:adjacency tensor} or equivalently the nonnegative eigenvalues of the following 
block matrix 
\[
\begin{pmatrix}
0 & A\\
A^{\mathrm{T}} & 0   
\end{pmatrix}.
\]

\end{document}